\documentclass [11pt,oneside,a4paper,mathscr]{amsart}

\usepackage{amscd,amsmath,amssymb,euscript}
\usepackage[frame,cmtip,curve,arrow,matrix,line,graph]{xy}
\usepackage{tikz-cd}
\usepackage{bigints}
\usetikzlibrary{matrix}
\usepackage{hyperref}
\usepackage{stmaryrd}
\usepackage[shortlabels]{enumitem}
\usepackage{here}

\title[Categorical DT/PT correspondence]{The local categorical DT/PT correspondence}

\author{Tudor P\u adurariu and Yukinobu Toda}

\newtheorem{thm}{Theorem}[section]
\newtheorem{cor}[thm]{Corollary}

\newtheorem{prop}[thm]{Proposition}
\newtheorem{lemma}[thm]{Lemma}

\theoremstyle{definition}

\newtheorem{thm*}[thm]{Theorem$^*$}
\newtheorem{remark}[thm]{Remark}

\newtheorem{step}{Step}
\newcommand{\relmiddle}[1]{\mathrel{}\middle#1\mathrel{}}
\newcommand{\comment}[1]{}

\renewcommand{\leq}{\leqslant}
\renewcommand{\geq}{\geqslant}

\newcommand{\OO}{\mathcal{O}}

\newcommand{\X}{\mathcal{X}}

\newcommand{\Coh}{\operatorname{Coh}}

\newcommand{\Ker}{\operatorname{Ker}}
\newcommand{\id}{\operatorname{id}}
\newcommand{\Ext}{\operatorname{Ext}}
\newcommand{\Hom}{\operatorname{Hom}}

\newcommand{\ch}{\operatorname{ch}}

\newcommand{\dd}{\underline{d}}

\newcommand{\Tr}{\mathop{\mathrm{Tr}}\nolimits}

\newcommand{\ssslash}{/\!\!/}

\usetikzlibrary{positioning,fit,calc}
\tikzstyle{block}=[draw=black, width=1cm, minimum height=2cm, align=center] 
\tikzstyle{block2}=[draw=black, text width=2cm, minimum height=1cm, align=center] 
\tikzstyle{block3}=[draw=black, text width=2cm, minimum height=1cm, align=center] 

\makeatletter

\@addtoreset{equation}{section}	
\makeatother

\begin{document}
\maketitle
\begin{abstract}
In this paper, we prove the categorical wall-crossing formula for certain quivers containing the three loop 
quiver, which we call DT/PT quivers. 
These quivers appear as Ext-quivers for the wall-crossing of DT/PT moduli spaces on Calabi-Yau 3-folds. 
The resulting formula is a semiorthogonal decomposition which involves quasi-BPS 
categories studied in our previous papers, and we regard it as a categorical analogue of the numerical DT/PT 
correspondence. 
As an application, we prove a categorical DT/PT correspondence for sheaves supported on reduced plane curves in the affine three dimensional space.   
\end{abstract}

\section{Introduction}
The purpose of this paper is to give a 
categorical wall-crossing formula 
for certain quivers, called \textit{DT/PT quivers}. 
Our main result is motivated 
by our pursuit of categorifying the DT/PT correspondence 
for curve counting theories on Calabi-Yau 3-folds, 
see~Subsection~\ref{sub:dtpt} for its review. 
As we will explain in Subsection~\ref{subdtpt2}, 
the DT/PT quivers appear as 
Ext-quivers of wall-crossing of DT/PT moduli 
spaces on Calabi-Yau 3-folds. We regard the 
resulting categorical wall-crossing formula as a local categorical analogue of the
DT/PT correspondence. 
We 
apply it to obtain a categorical (and K-theoretic) DT/PT correspondence in a geometric example, for sheaves supported on reduced plane curves in the affine three dimensional space. 
The main result 
of this paper will be applied in~\cite{PT2} 
to give a categorical DT/PT wall-crossing 
for reduced curve 
classes on local surfaces.

\subsection{The categorical wall-crossing formula for DT/PT quivers}\label{localquivers}
For $a\in \mathbb{N}$, 
we define the DT/PT quiver 
$Q^{af}$ to be the quiver with two vertices $0$ and $1$, three loops at $1$, $a+1$ edges from $0$ to $1$, and $a$ edges from $1$ to $0$, see Figure~1 for the picture of $Q^{2f}$. 
\begin{figure}[H]
\begin{align*}
	Q^{2f}
	\begin{tikzpicture}			
			\draw[->_>_>, >={Latex[round]}] 	
			(-4, 0) to [bend left=30] (0, 0);
					\draw[->_>, >={Latex[round]}] 	
				(0, 0) to [bend left=30] (-4, 0);
				\draw[->, >={Latex[round]}] (0, 0) arc (-180:0:0.4) ;
		\draw (0.8, 0) arc (0:180:0.4);
		\draw[->, >={Latex[round]}] (0, 0) arc (-180:0:0.6) ;
		\draw (1.2, 0) arc (0:180:0.6);
		\draw[->, >={Latex[round]}] (0, 0) arc (-180:0:0.8);
		\draw (1.6, 0) arc (0:180:0.8);
		\draw[fill=black] (0, 0) circle (0.05);
		\draw[fill=black] (-4, 0) circle (0.05);
	\end{tikzpicture}
	\end{align*}
	\caption{DT/PT quiver $Q^{af}$ for $a=2$}
	\end{figure}
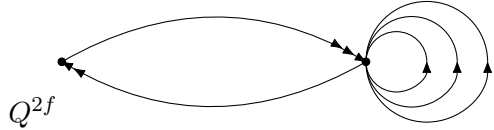
We denote by $R^{af}(1, d)$ the affine space of representations of $Q^{af}$ of dimension $(1, d)$.
Let $\chi_0 \colon GL(V) \to \mathbb{C}^{\ast}$ be the determinant character
$g \mapsto \det g$. Consider the GIT semistable stacks, which are smooth 
quasi-projective varieties: 
\begin{align*}
    I^a(d)&:=R^{af}(1, d)^{\chi_0\text{-ss}}/GL(d),\\
    P^a(d)&:=R^{af}(1, d)^{\chi^{-1}_0\text{-ss}}/GL(d).
\end{align*}
As we explain in Subsection~\ref{subdtpt2}, these moduli spaces together 
with some super-potentials give local models of the DT/PT moduli spaces on 
Calabi-Yau 3-folds. 
The following is the main result
of this paper: 
\begin{thm}\label{thm2bis}
    Let $\mu\in\mathbb{R}$
    such that $2\mu l \notin \mathbb{Z}$
    for $1\leq l \leq d$. 
    There is a semiorthogonal decomposition 
    \begin{equation}\label{SOD11}
    D^b(I^a(d))=\Big\langle  \left( \boxtimes_{i=1}^k \mathbb{M}(d_i)_{w_i}\right)\boxtimes D^b(P^a(d'))\Big\rangle.
   \end{equation}
    The right hand side is after all $d'\leq d$, partitions $(d_i)_{i=1}^k$ of $d-d'$, and integers $(w_i)_{i=1}^k$ such that for $v_i:=w_i+d_i\left(d'+\sum_{j>i} d_j-\sum_{j<i}d_j\right)$, we have
    \begin{equation}\label{boundsalpha}
    -1-\mu-\frac{a}{2}< \frac{v_1}{d_1}<\cdots<\frac{v_k}{d_k}<-\mu-\frac{a}{2}.
\end{equation}
    The order on the categories is discussed in Subsection \ref{comparison}, see also Subsection \ref{sod}.
\end{thm}

In the above, the categories $\mathbb{M}(d)_w$ are certain (twisted) noncommutative resolutions of singularities of $\mathfrak{gl}(d)^{\oplus 3}\ssslash GL(d)$ introduced by \v{S}penko--Van den Bergh \cite{SVdB}.

We treated the case $a=0$ in \cite[Theorem 1.1]{PT0}, when $P^0(d)$ is empty for $d>0$ and is a point for $d=0$. We note that while the proof of \cite[Theorem 1.1]{PT0} is based on the constructions of \cite{P} (going back to \cite{MR4227868}) and on the use of decompositions of the weight space of $T(d)\cong (\mathbb{C}^*)^d$ in polytopes using the $(r, p)$-invariant of a weight introduced by \v{S}penko--Van den Bergh \cite{SVdB, MR4227868}, a similar argument does not produce the semiorthogonal decomposition \eqref{SOD11}.
The proof of Theorem \ref{thm2bis} is still based on decompositions of polytopes in the weight space of $T(d)$, but these decompositions are different from the ones considered in \cite[Corollary 8.5]{MR4227868}, \cite[Theorem 1.1]{P}, \cite[Theorem 1.1]{PT0}.

Using \cite[Theorem 1.1]{PT0}, we also obtain the following decomposition in K-theory, where $K$ denotes the Grothendieck group of a dg-category: 
\begin{cor}\label{thm12}
Let $d\in \mathbb{N}$. Then there is an isomorphism \begin{align*}
    K(I^a(d))\cong\bigoplus_{d'=0}^d 
    K(\mathrm{NHilb}(d-d')) \otimes
    K(P^a(d')). 
\end{align*}
\end{cor}
Compare the above isomorphism with  \eqref{DTPT}. 
We prove Theorem \ref{thm2bis} and Corollary \ref{thm12} in Section \ref{sec:DT/PTlocal}.

\subsection{The DT/PT correspondence 
for Calabi-Yau 3-folds}\label{sub:dtpt}
We now review the DT/PT correspondence for 
Calabi-Yau 3-folds. 
For a smooth Calabi-Yau 3-fold $X$, let $\beta \in H_2(X, \mathbb{Z})$ and 
$n \in \mathbb{Z}$. 
The moduli space which defines the 
(rank one) Donaldson-Thomas (DT) invariant $I_{n, \beta}$ is the 
classical Hilbert scheme
\begin{align*}
	I_X(\beta, n)
	\end{align*}
which parametrizes closed subschemes $C \subset X$ with 
$\dim C \leq 1$
and $([C], \chi(\mathcal{O}_C))=(\beta, n)\in H_2(X, \mathbb{Z})\oplus \mathbb{Z}.$
The rank one DT invariant 
is defined by 
\begin{align*}
    \mathrm{DT}_{\beta, n}=\int_{I_X(\beta, n)}
    \chi_B \ de,
\end{align*}
where $\chi_B$ is the Behrend constructible 
function \cite{Beh}.
The DT invariants are
related to Gromov-Witten invariants by the 
Maulik-Nekrasov-Okounkov-Pandharipande (MNOP) 
conjecture \cite{MNOP}. 

In~\cite{MR2545686}, Pandharipande--Thomas 
defined virtual counts of stable pairs on a smooth complex 3-fold, which we call PT invariants. In the Calabi-Yau case, PT invariants provide a geometric construction of the contribution of DT invariants in the statement of the MNOP conjecture.
The moduli space of PT stable pairs 
is denoted by 
\begin{align*}
	P_X(\beta, n)
	\end{align*}
and it 
parametrizes pairs $(F, s)$, where $F$ is 
a pure one-dimensional coherent sheaf on $X$ 
with $([F], \chi(F))=(\beta, n)\in H_2(X, \mathbb{Z})\oplus \mathbb{Z}$ and 
$s \colon \mathcal{O}_X \to F$ is a section 
whose cokernel is at most zero-dimensional. 
The PT invariant is 
defined by 
\begin{align*}
  \mathrm{PT}_{\beta, n}=\int_{P_X(\beta, n)}
    \chi_B \ de. 
\end{align*}
 The DT/PT correspondence 
is the formula 
\begin{equation}\label{DTPT}
    \text{DT}_{\beta, n}=\sum_{k\geq 0}\text{DT}_{0, k}\cdot\text{PT}_{\beta, n-k}.
\end{equation}

The above formula was conjectured in~\cite{MR2545686}, 
its $\chi_B$-unweighted version 
was proved in~\cite{MR2669709, MR2869309}
using Joyce's motivic Hall algebra machinery~\cite{Joy1, Joy2, Joy3, Joy4}, and its version incorporating the Behrend function was proved in~\cite{Br} (see also~\cite{Thall})
using the work of 
Joyce-Song~\cite{JS}
(see also~\cite{K-S}).

There are several refinements of DT invariants in the Calabi-Yau case: motivic, cohomological, K-theoretic, or categorical, see \cite{Sz, T, PT0}. The motivic DT/PT correspondence was established in local cases by Davison--Ricolfi \cite{MR4332481}. It is an interesting problem to study the cohomological DT/PT correspondence for quivers with potential (following \cite{DM}), or for global geometries.

In~\cite[Subsection 1.3, Conjecture 1.3.2]{T}, the second author 
proposed the existence of dg-categories
 $\mathcal{DT}_X(\beta, n)$ and
 $\mathcal{PT}_X(\beta, n)$
which categorify DT and PT invariants, and 
the existence of a fully-faithful functor 
\begin{align}\label{cat:DTPT}
	\mathcal{PT}_X(\beta, n)\hookrightarrow 
\mathcal{DT}_X(\beta, n).
\end{align}
We do not know how to construct such categories for a general Calabi-Yau $3$-fold.
In~\cite[Sections 3 and 4]{T}, the second author constructed such dg-categories 
in the case of local surfaces 
$X=\text{Tot}_S(\omega_S)$ (for $S$ a smooth surface), and proved the 
existence of a fully-faithful functor (\ref{cat:DTPT}) when 
$\beta$ is a reduced curve class \cite[Theorem 5.5.5]{T}. 
It is natural to try to obtain a finer statement than \eqref{cat:DTPT} such as a semiorthogonal 
decomposition of $\mathcal{DT}_X(\beta, n)$ which categorifies the numerical DT/PT correspondence \eqref{DTPT}. 

As we explain in the next subsection, 
the motivation of the main result in this paper 
is to establish such a categorical DT/PT wall-crossing 
for certain local models of the (still to be defined) dg-categories in (\ref{cat:DTPT}) via Ext-quivers.


\subsection{DT/PT wall-crossing and Ext-quivers}\label{subdtpt2}
We now explain a construction of local models of the dg-categories (\ref{cat:DTPT})
in terms of Ext-quivers associated with a wall-crossing diagram. 
For a collection of objects $\{E_1, \ldots, E_k\}$, 
the associated Ext-quiver has set of vertices $\{1, \ldots, k\}$
and the number of arrows from $i$ to $j$ is 
$\dim \Ext^1(E_i, E_j)$. 

The two moduli spaces $I_X(\beta, n)$ and $P_X(\beta, n)$ are related 
by wall-crossing in the derived category of $X$ by 
regarding them as moduli spaces of certain two-term complexes. 
We have open immersions 
\begin{align*}
	I_X(\beta, n) \subset \mathcal{T}_X(\beta, n)
	\supset P_X(\beta, n),
	\end{align*}
where 
	 $\mathcal{T}_X(\beta, n)$ is the moduli stack of 
pairs $(F, s)$ such that $F$ is a (not necessary pure) one-dimensional 
sheaf and $s \colon \mathcal{O}_X \to F$ is a section 
with at most zero-dimensional cokernel. 
The stack $\mathcal{T}_X(\beta, n)$ is the moduli 
stack of semistable objects on the DT/PT wall. 
Consider its good moduli space 
\begin{align}\label{intro:good}
	\pi \colon 
\mathcal{T}_X(\beta, n) \to T_X(\beta, n).
\end{align}
We have the following wall-crossing diagram 
\begin{align*}
	\xymatrix{
I_X(\beta, n) \ar[dr]_-{\pi^{+}} & & P_X(\beta, n) \ar[ld]^-{\pi^-} \\
& T_X(\beta, n), &	
}
	\end{align*}
which is an example of a d-critical flip as defined in~\cite{Toddbir}. 

A closed point $p \in T_X(\beta, n)$ corresponds to a 
direct sum 
\begin{align*}
	I_C \oplus \bigoplus_{j=1}^m V^{(j)}\otimes \mathcal{O}_{x^{(j)}}[-1]
	\end{align*}
where $C \subset X$ is a Cohen-Macaulay curve 
and $x^{(1)}, \ldots, x^{(m)}$ are distinct points in $X$. 
The formal fiber $\widehat{\mathcal{T}}_X(\beta, n)_p$
of the morphism (\ref{intro:good}) at $p$ 
is described in terms of the Ext-quiver $Q_p$
associated with 
the collection 
\begin{align*}
	\{I_C, \mathcal{O}_{x^{(1)}}[-1], \ldots, \mathcal{O}_{x^{(m)}}[-1]\}.
	\end{align*}
The Ext-quiver $Q_p$ is given by gluing the quivers $Q^{a^{(j)}f}$ at
the vertex $0$, and adding $N$-loops 
at $0$, 
where $a^{(j)}=\mathrm{ext}^1(I_C, \mathcal{O}_{x^{(j)}})$
and $N=\mathrm{ext}^1(I_C, I_C)$, see \cite[Proposition 5.5.2]{T}. 
The moduli stack of the representations of the 
above Ext-quiver $Q_p$ with dimension vector $(1, d^{(1)}, \ldots, d^{(m)})$, 
where $d^{(j)}=\dim V^{(j)}$, is then given by 
\begin{align*}
\mathcal{U}_p=\left(\mathbb{C}^N \times \prod_{j=1}^m R^{a^{(j)}f}(1, d^{(j)})
   \right)\Big/\prod_{j=1}^m GL(d^{(j)}). 
	\end{align*}
Then there is a super-potential $w_p$ on 
some formal completion $\widehat{\mathcal{U}}_p$
of $\mathcal{U}_p$ such that 
$\widehat{\mathcal{T}}_X(\beta, n)_p$
is isomorphic to $\mathrm{Crit}(w_p)$, see \cite{Todstack}, \cite[Proposition 9.11]{Toddbir}. 

Let $\chi_0$ be the character of $\prod_{j=1}^k GL(d^{(j)})$
given by the product of each determinant character
of $GL(d^{(j)})$. 
Then the formal fibers of $\pi^{\pm}$ are given by 
$\chi_0^{\pm}$-semistable loci 
on $\mathrm{Crit}(w_p)$. 
Therefore 
the categories 
\begin{align*}
	\widehat{\mathcal{DT}}_X(\beta, n)_p:=
	\mathrm{MF}(\widehat{\mathcal{U}}_p^{+}, w_p), \ 
	\widehat{\mathcal{PT}}_X(\beta, n)_p:=\mathrm{MF}(\widehat{\mathcal{U}}_p^-, w_p)
	\end{align*}
are natural candidates of formal local models of 
the (yet to be defined) DT/PT categories in (\ref{cat:DTPT}).
Theorem~\ref{thm2bis}
(more precisely, its version allowing loops at $0$ and a super-potential, see~Corollary~\ref{cor:variant})
implies that there is a semiorthogonal decomposition
\begin{align}\label{sod:formal}
		\widehat{\mathcal{DT}}_X(\beta, n)_p=
		\left\langle \boxtimes_{j=1}^m \boxtimes_{i=1}^{k^{(j)}}
		\widehat{\mathbb{S}}(d_i^{(j)})_{w_i^{(j)}} \boxtimes 
		\widehat{\mathcal{PT}}_X(\beta, n')_{p'}
		  \right\rangle,
	\end{align}
where $\sum d_{i}^{(j)}+n'=n$,  
the weights $w_i^{(j)}$ satisfy a condition analogous to \eqref{boundsalpha}, 
and the point $p' \in T_{n'}(X, \beta)$ satisfies 
$\sum d_i^{(j)}x^{(j)}+p'=p$. 
In the above, $\widehat{\mathbb{S}}(d)_w$ is the
formal analogue of the quasi-BPS 
category $\mathbb{S}(d)_w$ of $\mathbb{C}^3$ considered in~\cite{PT0, PT1}. 

We expect that (\ref{sod:formal}) can be globalized to give a 
semiorthogonal decomposition of $\mathcal{DT}_X(\beta, n)$. 
An issue here is that (\ref{sod:formal}) depends on local data $a^{(j)}$ 
and a choice of $\mu$ in Theorem~\ref{thm2bis}. 
In Theorem~\ref{thm2bis}, if we take $\mu=-a/2-\varepsilon$
for $0<\varepsilon \ll 1$, 
then these data cancel and the condition (\ref{boundsalpha}) becomes
\begin{align}\label{boundsalpha1.5}
    -1< \frac{v_1}{d_1}<\cdots<\frac{v_k}{d_k}\leq 0. 
\end{align}
From the above consideration, 
if the DT/PT categories in (\ref{cat:DTPT}) exist, we expect a 
semiorthogonal decomposition of the form 
\begin{align}\label{DTPT:global}
	\mathcal{DT}_X(\beta, n)=\left\langle 
	\boxtimes_{i=1}^k\mathbb{S}_X(d_i)_{w_i} \boxtimes 
	\mathcal{PT}_X(\beta, n')
	\right\rangle
	\end{align}
where $d_1+\cdots+d_k+n'=n$, 
the integers $w_i$ satisfy the condition (\ref{boundsalpha1.5})
under the transformation $w_i \mapsto v_i$ in Theorem~\ref{thm2bis}, 
and $\mathbb{S}_X(d)_w$ is an
analogue of the quasi-BPS category $\mathbb{S}(d)_w$ for $X$. 

In the local surface case, the DT/PT 
categories have been defined in~\cite{T} and the quasi-BPS categories in~\cite{P2}. 
In \cite{PT2},
we use the results of this paper to
construct the semiorthogonal 
decomposition (\ref{DTPT:global}) for reduced curve classes on local surfaces.

\subsection{The categorical DT/PT correspondence for \texorpdfstring{$\mathbb{C}^3$}{C3}}

We will give 
an example of dg-categories
of matrix factorizations with 
explicit super-potentials
which categorify DT/PT invariants on 
$\mathbb{C}^3$ with reduced supports. 
Then Theorem \ref{thm2bis} for $a=1$ 
implies the categorical DT/PT correspondence 
in this case (not only 
formal locally as in Subsection~\ref{sub:dtpt}). We are able to define the categories $\mathcal{DT}$ and $\mathcal{PT}$ in this case because of the global description as a critical locus of the relevant moduli spaces, see Theorem \ref{thm:globalcrit} and the discussion in Subsection \ref{sub:45}.

Let 
$X=\mathrm{Tot}_{\mathbb{P}^2}(\omega_{\mathbb{P}^2})$, and let 
$
\mathcal{T}_X^{\rm{red}}(m, d)
$ be the classical moduli stack of pairs $(F, s)$, where $F$
is a one-dimensional sheaf on $X$ with support a reduced plane curve
of degree $m$ 
in $\mathbb{P}^2$, and $s \colon \mathcal{O}_X \to F$ is a section with at most zero-dimensional cokernel. 
The open immersion $\mathbb{C}^2 \subset \mathbb{P}^2$ determines 
an open immersion $\mathbb{C}^3 \subset X$. 
We consider the open substack 
\begin{align*}
	\mathcal{T}_{\mathbb{C}^3}^{\rm{red}}(m, d)
	\subset \mathcal{T}_X^{\rm{red}}(m, d)
\end{align*} consisting of $(F, s)$ such that 
$\mathrm{Cok}(s)$ and the maximal zero-dimensional subsheaf of $F$ are supported on $\mathbb{C}^3$, 
and the one-dimensional support of $F$ is a plane curve from $(\mathbb{C}^{I_m})^{\rm{red}}$, see \eqref{open:I} and \eqref{ired2}. 

\begin{thm}\emph{(Theorem~\ref{lemma:globalcritical})}\label{thm:globalcrit}
Let $V$ be a vector space of dimension $d$.
Then $\mathcal{T}_{\mathbb{C}^3}^{\rm{red}}(m, d)$ is the global critical locus of a regular function \[\mathcal{T}_{\mathbb{C}^3}^{\rm{red}}(m, d)=\mathrm{Crit}\left(\Tr W_{m, d}\colon \left(V^{\oplus 2} \oplus V^{\vee} \oplus \mathfrak{g}^{\oplus 3} \times (\mathbb{C}^{I_m})^{\rm{red}}\right)\big/GL(V) \to\mathbb{C}\right),\]
where $\Tr W_{m, d}$ is explicitly 
given by, see~\eqref{def:falpha} and \eqref{W:formula}:
\begin{align*}
    \Tr W_{m, d}(u_1, u_2, v, A, B, C, (\alpha_{ij}))
     =\sum_{1\leq i+j \leq m}\alpha_{ij} v A^i B^j(u_2)
     +\Tr C(u_1 \circ v+[A, B]). 
\end{align*}
\end{thm}
Using the above presentation of $\mathcal{T}_{\mathbb{C}^3}^{\rm{red}}(m, d)$ as a global critical locus, 
we define dg-categories
\[\mathcal{DT}_{\mathbb{C}^3}^{\rm{red}}(m, d), \ (\text{resp.}~\mathcal{PT}_{\mathbb{C}^3}^{\rm{red}}(m, d))\] categorifying DT (resp.~PT) invariants
	for 
	the Hilbert schemes of 1-dimensional subschemes (resp.~PT stable pair 
	moduli spaces) with reduced 
	supports. 
	Using Theorem \ref{thm2bis} for $a=1$ 
	and $\mu=-1/2-\varepsilon$ for $0<\varepsilon \ll 1$
	and applying matrix factorizations for $\Tr W_{d,m}$, 
	see Corollary~\ref{cor:variant}, 
	we obtain the following:
	
		\begin{thm}\label{thm:exam}
	There is a semiorthogonal decomposition 
		\begin{align}\label{SODlocalC3}
			\mathcal{DT}_{\mathbb{C}^3}^{\rm{red}}(m, d)
			=\left \langle \left(\boxtimes_{i=1}^k
			\mathbb{S}(d_i)_{w_i}
			 \right) \boxtimes \mathcal{PT}_{\mathbb{C}^3}^{\rm{red}}(m, d')
		 \right\rangle. 
			\end{align}
			The right hand side is after all $d'\leq d$, partitions $(d_i)_{i=1}^k$ of $d-d'$, and integers $(w_i)_{i=1}^k$ such that for $v_i:=w_i+d_i\left(d'+\sum_{j>i} d_j-\sum_{j<i}d_j\right)$, we have
    \begin{equation*}
    -1< \frac{v_1}{d_1}<\cdots<\frac{v_k}{d_k} \leq 0.
\end{equation*}
		\end{thm}
The quasi-BPS categories $\mathbb{S}(d)_w$ are defined as categories of matrix factorizations for the potential $\mathrm{Tr}\,X[Y,Z]$ with factors in $\mathbb{M}(d)_w$, see \cite{PT1} for results on these categories. Fix $0\leq d'\leq d$.
By \cite[Theorem 1.1]{PT0}, the categories with which we tensor $\mathcal{PT}_{\mathbb{C}^3}^{\rm{red}}(m, d')$ in \eqref{SODlocalC3} form a semiorthogonal decomposition of $\mathcal{DT}(d-d')$, the DT category of $d-d'$ points in $\mathbb{C}^3$. We thus obtain the following decomposition in K-theory, compare with \eqref{DTPT}:
\begin{cor}\label{cor14}
Let $d\in \mathbb{N}$. Then there is an isomorphism \begin{align*}
    K\left(\mathcal{DT}_{\mathbb{C}^3}^{\rm{red}}(m, d)\right)\cong
    \bigoplus_{d'=0}^d K\left(\mathcal{DT}(d-d')\otimes \mathcal{PT}_{\mathbb{C}^3}^{\rm{red}}(m, d')\right).
\end{align*}
\end{cor}

We do not know whether the Künneth isomorphism holds above. In Corollary \ref{cor15}, we discuss the analogous statement in the localized equivariant situation (with respect to the two dimensional Calabi-Yau torus of $\mathbb{C}^3$), when the Künneth isomorphism holds.

\subsection{Acknowledgements}
	Y.~T.~is supported by World Premier International Research Center
	Initiative (WPI initiative), MEXT, Japan, and Grant-in Aid for Scientific
	Research grant (No.~19H01779) from MEXT, Japan.

\section{Preliminaries}

\subsection{Notations}\label{subsection:notation}
The spaces considered in this paper are defined 
over the complex field $\mathbb{C}$ and are quotient stacks $\X=A/G$, where $A$ is a dg-scheme (or also 
called derived scheme), the derived zero locus of a section $s$ of a finite rank vector bundle $\mathcal{E}$ on a finite type separated scheme $X$ over $\mathbb{C}$, and $G$ is a reductive group. For such a dg-scheme $A$, let $\dim A:=\dim X-\text{rank}(\mathcal{E})$,
let $A^{\mathrm{cl}}:=Z(s)\subset X$ be the (classical) zero locus, and let $\X^{\mathrm{cl}}:=A^{\mathrm{cl}}/G$. We denote by $\mathcal{O}_{\X}$ or $\mathcal{O}_A$ the structure sheaf of $\X$. 
We denote by $D^b(\mathcal{X})$ the bounded derived category of coherent sheaves on $\X$. 

For $G$ a reductive group and $A$ a dg-scheme as above, denote by $A/G$ the corresponding quotient stack and by $A\ssslash G$ the quotient dg-scheme with dg-ring of regular functions $\mathcal{O}_A^G$.

In this paper, all the dg-categories are $\mathbb{C}$-linear 
pre-triangulated dg-categories, in particular their homotopy categories are 
triangulated categories. 
For $\mathcal{D}$ a dg-category, we denote by $K(\mathcal{D})$ the Grothendieck group of the homotopy category of $\mathcal{D}$. 
For $\X$ a stack as above, we denote by $G(\X)=G(\X^{\text{cl}})$ the Grothendieck group of $D^b\text{Coh}(\X)$ and by $K(\X)$ the Grothendieck group of the category of perfect complexes $\text{Perf}(\X)\subset D^b(\X)$.

\subsection{Weights and partitions} 

\subsubsection{}\label{sss1}
Let $Q=(I, E)$ be the quiver with vertex set $I=\{1\}$ and edge set $E=\{x, y, z\}$. 
For $a\in \mathbb{N}$, let $Q^a=(J, E^a)$ be the quiver with vertex set $J=\{0, 1\}$ and edge set $E^a$ containing three loops $E=\{x, y, z\}$ at $1$, $a$ edges from $0$ to $1$, and $a$ edges from $1$ to $0$. Then $Q^0$ is the disjoint union of $Q$ and the quiver with one vertex $0$.
Let $Q^{af}=(J, E^{af})$ be the quiver with edges 
$E^{af}=e \sqcup E^a$, where $e$ is an edge from $0$ to $1$, see Figure~1 for a picture of $Q^{2f}$. 

For $d\in \mathbb{N}$, let $V$ be a $\mathbb{C}$-vector space of dimension $d$. 
We often write $GL(d)$ as $GL(V)$. 
Its Lie algebra is 
denoted by $\mathfrak{gl}(d)=\mathfrak{gl}(V):=\text{End}(V)$. 
When the dimension is clear from the context, we drop $d$ from its notation
and write it as $\mathfrak{g}$. 
Consider the $GL(V)\cong GL(d)$ representations:
\begin{align*}
R(d)&:=\mathfrak{gl}(V)^{\oplus 3},\\
    R^a(1,d)&:=V^{\oplus a}\oplus \left(V^{\vee}\right)^{\oplus a}\oplus \mathfrak{gl}(V)^{\oplus 3},\\
    R^{af}(1,d)&:=V^{\oplus \left(a+1\right)}\oplus \left(V^{\vee}\right)^{\oplus a}\oplus \mathfrak{gl}(V)^{\oplus 3}.
\end{align*}
Define the following stacks:
\begin{align*}
\X(d)&:=R(d)/GL(d),\\
\X^a(1, d)&:=R^a(1, d)/GL(d),\\
\mathcal{X}^{af}(1, d)&:=R^{af}(1, d)/GL(d).
\end{align*}

 \subsubsection{}
 Fix 
 $T(d)\subset GL(d)$ the maximal torus consisting of diagonal matrices. 
Denote by $M(d)$ the weight space of $T(d)$ and let $M(d)_{\mathbb{R}}:=M(d)\otimes_{\mathbb{Z}}\mathbb{R}$. Let $\beta_1,\ldots, \beta_d$ be the simple roots of $GL(d)$. 
A weight $\chi=\sum_{i=1}^d c_i\beta_i$ is dominant 
(resp. strictly dominant) 
if 
\begin{align*}
c_1\leq\cdots\leq c_d, \ 
(\mbox{resp.~}
c_1<\cdots<c_d).
\end{align*}
We denote by $M^+\subset M$ and $M^+_{\mathbb{R}}\subset M_{\mathbb{R}}$ the dominant chambers. When we want to emphasize the dimension vector, we write $M(d)$ etc. Denote by $N$ the coweight lattice of $T(d)$ and by $N_{\mathbb{R}}:=N\otimes_{\mathbb{Z}}\mathbb{R}$. Let $\langle\,,\,\rangle$ be the natural pairing between $N_{\mathbb{R}}$ and $M_{\mathbb{R}}$. 

Let $W=\mathfrak{S}_d$ be the Weyl group of $GL(d)$. For $\chi\in M(d)^+$, let $\Gamma_{GL(d)}(\chi)$ be the irreducible 
representation of $GL(d)$ of highest weight $\chi$. We drop $GL(d)$ from the notation if the dimension vector $d$ is clear from the context. Let $w*\chi:=w(\chi+\rho)-\rho$ be the Weyl-shifted action of $w\in W$ on $\chi\in M(d)_\mathbb{R}$. We denote by $\ell(w)$ the length of $w\in W$.

\subsubsection{}\label{def:setW}
Denote by $\mathcal{W}^a$ the multiset of $T(d)$-weights of $R^a(1, d)$ and by $\mathcal{W}^{af}$ the multiset of $T(d)$-weights of $R^{af}(1, d)$.
Namely, we have 
\begin{align*}
\mathcal{W}^a=\{(\beta_i-\beta_j)^{\times 3}, (\pm \beta_i)^{\times a} \mid 1\leq i, j \leq d\}, \
\mathcal{W}^{af} =\mathcal{W}^a \cup \{\beta_i \mid 1\leq i\leq d\}. 
\end{align*}
We denote by 
$\rho$ half the sum of positive roots of $GL(d)$. 
In our convention of the dominant chamber, 
it is given by 
\begin{align*}
    \rho=\frac{1}{2}\mathfrak{g}^{\lambda<0}=\frac{1}{2}\sum_{j<i}(\beta_i-\beta_j),
\end{align*}
where $\lambda$ is the antidominant cocharacter $\lambda(t)=(t^d, t^{d-1}, \ldots, t)$. 
We denote by $1_d:=z\cdot\text{Id}$ the diagonal cocharacter of $T(d)$. 
We define the weights in $M_{\mathbb{R}}$:
\begin{align}\label{def:taud}
\sigma_d:=\sum_{j=1}^d\beta_j, \ \tau_d:=\frac{1}{d}\sigma_d. 
\end{align}


\subsubsection{}\label{ss13} Let $G$ be a reductive group (in this paper, one may assume $G=GL(d)$), let $X$ be a $G$-representation, and let
$\X=X/G$ be the corresponding quotient stack. Let $\mathcal{V}$ be the multiset of $T(d)$-weights of $X$.
For a cocharacter $\lambda$ of $T(d)$,
let $X^\lambda\subset X$ be the subspace generated by weights $\beta\in \mathcal{V}$ such that $\langle \lambda, \beta\rangle=0$, let $X^{\lambda\geq 0}\subset X$ be the subspace generated by weights $\beta\in \mathcal{V}$ such that $\langle \lambda, \beta\rangle\geq 0$, and let $G^\lambda$ and $G^{\lambda\geq 0}$ be the Levi and parabolic groups associated to $\lambda$.
Consider the fixed and attracting stacks
\begin{align}\notag
    \X^\lambda:=X^\lambda/ G^\lambda,\
    \X^{\lambda\geq 0}:=X^{\lambda\geq 0}/G^{\lambda\geq 0}
\end{align}
with maps
\begin{align}\label{map:attracting}\X^\lambda\xleftarrow{q_\lambda}\X^{\lambda\geq 0}\xrightarrow{p_\lambda}\X.
\end{align}
We abuse notation and denote by $\X^{\lambda\geq 0}$ the class \begin{align*}\left[X^{\lambda\geq 0}\right]-\left[\mathfrak{g}^{\lambda\geq 0}\right]\in K_0(T(d))=M\end{align*}
and by $\langle \lambda, \X^{\lambda\geq 0}\rangle$ the corresponding integer. We use the similar notation for $\X^{\lambda>0}$, $\X^{\lambda\leq 0}$ etc.

\subsubsection{}\label{paco} 
Let $d\in \mathbb{N}$.
We call $\dd:=(d_i)_{i=1}^k$ a partition of $d$ if $d_i\in\mathbb{N}$ are all non-zero and $\sum_{i=1}^k d_i=d$. We similarly define partitions of $(d,w)\in\mathbb{N}\times\mathbb{Z}$.
For a partition $(d_i)_{i=1}^k$
 of $d$, there is an antidominant cocharacter $\lambda$
 of $T(d)$ such that 
 $\mathcal{X}(d)^{\lambda}\cong \times_{i=1}^k \mathcal{X}(d_i)$. 
 For example we can take 
 \begin{align}\label{take:lambda}
     \lambda=(\overbrace{t^{k}, \ldots, t^{k}}^{d_1}, \overbrace{t^{k-1}, \ldots, t^{k-1}}^{d_2}, 
     \ldots, \overbrace{t, \ldots, t}^{d_k}). 
 \end{align}
 We have the maps from (\ref{map:attracting})
 \[\times_{i=1}^k\X(d_i)
\xleftarrow{q_\lambda}\X(d)^{\lambda\geq 0}
\xrightarrow{p_\lambda}\X(d).\]
We also use the notations $p_\lambda=p_{\dd}$, $q_\lambda=q_{\dd}$.
Conversely, given an antidominant cocharacter $\lambda$, there is an associated 
partition $(d_i)_{i=1}^k$ inducing the diagram above. 
 Define the length $\ell(\lambda):=k$. 

The stack $\mathcal{X}(d)^{\lambda \geq 0}$ is isomorphic to the 
moduli stack of filtrations of $Q$-representations 
\begin{align*}
    0=R_0 \subset R_1 \subset \cdots \subset R_k
\end{align*}
such that $R_i/R_{i-1}$ has dimension $d_i$. 
The morphism $q_{\lambda}$ sends the above filtration to its 
associated graded, and $p_{\lambda}$ sends it to $R_k$. 
The categorical Hall product for $Q$ is given 
by the functor 
$p_{\lambda*}q_\lambda^*=p_{\dd*}q_{\dd}^*$ and denoted by
\begin{align}\label{prel:hall}
    \ast 
     \colon D^b(\mathcal{X}(d_1)) \boxtimes
     \cdots \boxtimes D^b(\mathcal{X}(d_k))
     \to D^b(\mathcal{X}(d)). 
\end{align}
We may drop the subscript $\lambda$ or $\dd$ in the functors $p_*$ and $q^*$ when the cocharacter $\lambda$ or the partition $\dd$ is clear. 

We will also use the Hall products for the quivers $Q^a$ and $Q^{af}$. 
For $e\leq d$, let $(d_i)_{i=1}^k$ be a partition of $e$ and set $d'=d-e$. 
Let $\lambda$ be the antidominant cocharacter of $T(e)$ 
given by (\ref{take:lambda}), and 
set $\lambda'=(\lambda, 1_{d'})$
which is an antidominant cocharacter of $T(d)$. 
Then the diagrams of (\ref{map:attracting}) for $\mathcal{X}^a(1, d)$, 
$\mathcal{X}^{af}(1, d)$ 
are given by 
\begin{align}\label{attract:a}
&\times_{i=1}^k \X(d_i) \times \X^a(1, d') \stackrel{q_{\lambda}}{\leftarrow}
\X^a(1, d)^{\lambda \geq 0} \stackrel{p_{\lambda}}{\to} \X^a(1, d), \\ \notag
&\times_{i=1}^k \X(d_i) \times \X^{af}(1, d') \stackrel{q_{\lambda}}{\leftarrow}
\X^{af}(1, d)^{\lambda \geq 0} \stackrel{p_{\lambda}}{\to} \X^{af}(1, d)
\end{align}
respectively. 
The functors $p_{\lambda \ast}q_{\lambda}^{\ast}$ give 
categorical Hall products
\begin{align}\label{prel:halla}
    &\ast 
     \colon D^b(\mathcal{X}(d_1)) \boxtimes
     \cdots \boxtimes D^b(\mathcal{X}(d_k))\boxtimes D^b(\mathcal{X}^a(1, d'))
     \to D^b(\mathcal{X}^a(1, d)),\\
     &\ast 
     \colon D^b(\mathcal{X}(d_1)) \boxtimes
     \cdots \boxtimes D^b(\mathcal{X}(d_k))\boxtimes D^b(\mathcal{X}^{af}(1, d'))
     \to D^b(\mathcal{X}^{af}(1, d)).\nonumber
\end{align}

\subsubsection{}\label{id} Let $(d_i)_{i=1}^k$ be a partition of $d$. There is an identification \[\bigoplus_{i=1}^k M(d_i)\cong M(d),\] where the simple roots $\beta_j$ in $M(d_1)$ correspond to the first $d_1$ simple roots $\beta_j$ of $M(d)$ etc.





\subsection{Polytopes}\label{ss1}
We construct several polytopes in $M(d)_{\mathbb{R}}$ which will be used to define categories in Subsection \ref{ss2}. 
The polytope $\mathbf{W}(d)$ is defined as
\begin{equation}\label{W}
    \mathbf{W}(d):=\frac{3}{2}\text{sum}[0, \beta_i-\beta_j]+\mathbb{R}\tau_d\subset M(d)_{\mathbb{R}},
    \end{equation}
where the Minkowski sum is after all $1\leq i, j\leq d$ and where $\tau_d$ is given by (\ref{def:taud}).  Consider the hyperplane
\begin{equation}\label{W0}
    \mathbf{W}(d)_w:=\frac{3}{2}\text{sum}[0, \beta_i-\beta_j]+w\tau_d\subset \mathbf{W}(d).
    \end{equation}
    
The polytope $\mathbf{V}(d)$ is defined as
\begin{equation}\label{V}
    \mathbf{V}(d):=\frac{3}{2}\text{sum}[0, \beta_i-\beta_j]+\text{sum}[-\beta_k, 0]\subset M(d)_{\mathbb{R}},
\end{equation}
where the Minkowski sum is after all $1\leq i, j, k\leq d$. Note that the definition of the polytope $\mathbf{V}(d)$ differs by the one used in \cite{PT0} by a translation by $\sigma_d$. 
We let $\mathbf{V}(d)_w\subset \mathbf{V}(d)$ be the subspace of weights $\chi$ such that $\langle 1_d, \chi\rangle=w$.

Finally, define the (not necessarily closed) polytopes in $M(d)_\mathbb{R}$:
\begin{align*}
    \mathbf{W}^a(1, d)&:=\frac{3}{2}\text{sum}[0, \beta_i-\beta_j]+\frac{a}{2}\text{sum}(-\beta_k, \beta_k],
  \\
    \mathbf{V}^a(1, d)&:=\frac{3}{2}\text{sum}[0, \beta_i-\beta_j]+\frac{a}{2}\text{sum}[-\beta_k, \beta_k]+\text{sum}[-\beta_k, 0],
\end{align*}
where the Minkowski sums are after all $1\leq i, j, k\leq d$. 

\subsection{A corollary of the Borel-Weyl-Bott theorem}

For future reference, we state a result from \cite[Section 3.2]{hls}. We continue with the notations from Subsection \ref{ss13}. 
For a weight $\chi \in M$, 
let $\chi^+$ be the dominant Weyl-shifted conjugate of $\chi$ if it exists, and let $\chi^+=0$ otherwise.
Let $\mathcal{V}$ be the multiset of 
weights in $X$. 
For a multiset $J\subset \mathcal{V}$, let
\[\sigma_J:=\sum_{\beta\in J}\beta.\]
For a weight $\chi \in M$, let 
$w$ be the element of the Weyl group such that $w*(\chi-\sigma_J)$ is dominant or zero. It has length $\ell(w)=:\ell(J)$.

\begin{prop}\label{bbw}
Let $G$ be a reductive group, let $X$ be a $G$-representation, and
let $\lambda$ be a cocharacter of the maximal torus $T\subset G$. 
Recall the fixed and attracting stacks and the corresponding maps
\[X^\lambda/G^\lambda\xleftarrow{q_\lambda}X^{\lambda\geq 0}/G^{\lambda\geq 0}\xrightarrow{p_\lambda}X/G.\]
Let $\chi$ be a weight of $T$. 
Then there is a quasi-isomorphism
\[\left(\bigoplus_{J}\mathcal{O}_{X}\otimes \Gamma_{G}\left((\chi-\sigma_J)^+\right)\left[|J|-\ell(J)\right], d\right)\xrightarrow{\sim}p_{\lambda*}q_{\lambda}^*\left(\mathcal{O}_{X^\lambda}\otimes\Gamma_{G^\lambda}(\chi)\right),\] where the complex on the left hand side has terms (shifted) vector bundles for all multisets $J\subset \{\beta\in\mathcal{V} \mid \langle \lambda, \beta\rangle<0\}$. 
\end{prop}



\subsection{Semiorthogonal decompositions}\label{sod}
Let $R$ be a set. Consider a set $O\subset R\times R$ such that for any $i, j\in R$ we have $(i,j)\in O$, or $(j,i)\in O$, or both $(i,j)\in O$ and $(j,i)\in O$. 

Let $\mathbb{T}$ be a dg-category. We will construct semiorthogonal decompositions
\[\mathbb{T}=\langle \mathbb{A}_i \mid i \in R \rangle\] with summands 
dg-subcategories $\mathbb{A}_i$ indexed by $i\in R$
such that for any $i,j\in R$ with $(i, j)\in O$ and objects $\mathcal{A}_i\in\mathbb{A}_i$, $\mathcal{A}_j\in\mathbb{A}_j$, we have 
\[\Hom_{\mathbb{T}}(\mathcal{A}_i,\mathcal{A}_j)=0.
\]

\subsection{Matrix factorizations}\label{MF}
Let $\mathcal{X}$ be a smooth and let
$w \colon \mathcal{X} \to \mathbb{C}$ be a regular function. We denote by 
\begin{align*}
    \mathrm{MF}(\mathcal{X}, w)
\end{align*}
the category of matrix factorizations of $w$. 
It consists of objects 
$\left(\alpha \colon F\rightleftarrows G\colon \beta\right)$
with $F, G \in \mathrm{Coh}(\mathcal{X})$
such that both of $\alpha \circ \beta$
and $\beta \circ \alpha$ are multiplications by $w$. 
We refer to \cite[Subsection 2.6]{PT0} for 
details about categories of matrix factorizations.

\subsection{Categories of generators}\label{ss2}

\subsubsection{}\label{ss:Ddelta}
For $w \in \mathbb{Z}$, 
we denote by $D^b(\mathcal{X}(d))_w$
the subcategory of $D^b(\mathcal{X}(d))$
consisting of objects of
weight $w$ with respect to the diagonal 
cocharacter $1_d$ of $T(d)$.  
We have the direct sum decomposition 
\begin{align*}
    D^b(\mathcal{X}(d))=\bigoplus_{w\in \mathbb{Z}}
    D^b(\mathcal{X}(d))_w. 
\end{align*}
We define the dg-subcategories 
\begin{align*}
    \mathbb{M}(d) \subset D^b(\X(d)), \ 
    (\mbox{resp. }
    \mathbb{M}(d)_w \subset D^b(\X(d))_w)
\end{align*}
to be generated 
by the vector bundles $\OO_{\X(d)}\otimes \Gamma_{GL(d)}(\chi)$, where $\chi$ is a dominant weight of $T(d)$ such that
\begin{equation}\label{M}
    \chi+\rho\in \mathbf{W}(d), \ 
    (\mbox{resp. } \chi+\rho \in \mathbf{W}(d)_w). 
    \end{equation}
    Note that $\mathbb{M}(d)$
    decomposes into the direct sum of $\mathbb{M}(d)_w$
    for $w \in \mathbb{Z}$. 
    Moreover, taking the tensor product with 
    the determinant character $\det \colon GL(d) \to \mathbb{C}^{\ast}$
    gives an equivalence
    \begin{align}\label{equiv:periodic}
    \otimes \det \colon 
    \mathbb{M}(d)_w \stackrel{\sim}{\to} \mathbb{M}(d)_{d+w}. 
    \end{align}
    Let $\mu\in \mathbb{R}$ and let $\delta:=\mu\sigma_d\in M(d)_\mathbb{R}$.
    We define the dg-subcategories \begin{align*}\mathbb{D}(d; \delta) \subset D^b(\X(d)), \ (\mbox{resp. }\mathbb{D}(d; \delta)_w \subset D^b(\X(d))_w)\end{align*} generated by the vector bundles $\OO_{\X(d)}\otimes \Gamma_{GL(d)}(\chi)$, where $\chi$ is a dominant weight of $T(d)$ such that \begin{equation}\label{MM}\chi+\rho+\delta\in \mathbf{V}(d),\ (\mbox{resp. } \chi+\rho +\delta \in \mathbf{V}(d)_w).  \end{equation} Note that $\mathbb{D}(d; \delta)$ decomposes into the direct sum of $\mathbb{D}(d; \delta)_w$ for $w \in \mathbb{Z}$ defined as above. 
    We will use the following proposition proved in~\cite{PT0}:
    \begin{prop}\emph{(\cite[Proposition~3.9]{PT0})}\label{prop:sodD}
    There is a semiorthogonal decomposition 
    \begin{align*}
        \mathbb{D}(d; \delta)=\left\langle \boxtimes_{i=1}^k \mathbb{M}(d_i)_{w_i} \right\rangle. 
    \end{align*}
    Here, the right hand side is after all partitions $d_1+\cdots+d_k=d$
    such that, by setting 
    $v_i:=w_i+d_i(\sum_{j>i}d_j-\sum_{j<i}d_j)$,
    we have 
    \begin{align*}
        -1-\mu \leq \frac{v_1}{d_1}<\cdots<\frac{v_k}{d_k} \leq -\mu.
    \end{align*}
       Moreover, each fully-faithful functor 
       $\boxtimes_{i=1}^k \mathbb{M}(d_i)_{w_i} \to \mathbb{D}(d; \delta)$
       is given by the restriction of the categorical Hall product (\ref{prel:hall}). 
    \end{prop}

\subsubsection{}\label{defcategories}
We fix $\mu\in \mathbb{R}$
and set $\delta :=  \mu \sigma_d$. 
We define the dg-subcategories 
\begin{align*}
    \mathbb{M}^a(1, d; \delta)
    \subset D^b(\X^a(1, d)), \ 
    \mbox{(resp. } \mathbb{F}^a(1, d; \delta)
\subset D^b(\X^{af}(1, d)) )
    \end{align*} to be generated by the vector bundles $\OO_{\X^a(1,d)}\otimes \Gamma_{GL(d)}(\chi)$ (resp.~
    $\OO_{\X^{af}(1, d)}\otimes \Gamma_{GL(d)}(\chi)$), where $\chi$ is a dominant weight of $T(d)$ such that
\begin{equation}\label{m1d}
    \chi+\rho+\delta\in \mathbf{W}^a(1, d).
    \end{equation}
    We define the dg-subcategories
\begin{align}\label{subcat:D}
    \mathbb{D}^a(1, d; \delta)
    \subset D^b(\X^a(1, d)), \ 
    \mbox{(resp. } \mathbb{E}^a(1, d; \delta)
\subset D^b(\X^{af}(1, d)) )
    \end{align}
    to be generated by the vector bundles $\OO_{\X^a(1, d)}\otimes \Gamma_{GL(d)}(\chi)$
    (resp.~
    $\OO_{\X^{af}(1, d)}\otimes \Gamma_{GL(d)}(\chi)$),
    where $\chi$ is a dominant weight of $T(d)$ such that
\begin{equation}\label{D}
    \chi+\rho+\delta\in \mathbf{V}^a(1, d).
    \end{equation}
 Note that for the projection $b \colon \X^{af}(1, d) \to \X^a(1, d)$,
 the pull-back $b^{\ast}$ restricts to the functors 
 \begin{align*}
 b^{\ast} \colon \mathbb{M}^a(1, d; \delta) \to \mathbb{F}^a(1, d; \delta),\,
     b^{\ast} \colon \mathbb{D}^a(1, d; \delta) \to \mathbb{E}^a(1, d; \delta)
 \end{align*}
 whose essential images generate $\mathbb{F}^a(1, d; \delta)$ and $\mathbb{E}^a(1, d; \delta)$, respectively.

\subsubsection{}\label{gradingMF}
Let $\Tr W$ be the regular function \begin{equation*}\label{def:Wd}
    \Tr W:=\Tr X[Y, Z]\colon \X(d)=\mathfrak{gl}(d)^{\oplus 3}/GL(d)\to \mathbb{C}.
\end{equation*}
We define the subcategory
\[\mathbb{S}(d):=\text{MF}(\mathbb{M}(d), \Tr W)
\subset \mathrm{MF}(\mathcal{X}(d), \Tr W)
\] 
to be the subcategory of matrix factorizations
$\left(\alpha \colon F\rightleftarrows G\colon \beta\right)$ with $F$ and $G$ in $\mathbb{M}(d)$.
It decomposes into the direct sum of 
$\mathbb{S}(d)_w$ for $w \in \mathbb{Z}$, 
where $\mathbb{S}(d)_w$
is defined similarly to $\mathbb{S}(d)$
using $\mathbb{M}(d)_w$. 
There is an equivalence:
\begin{align}\label{equiv:periodic2}
    \otimes \det \colon 
    \mathbb{S}(d)_w \stackrel{\sim}{\to} \mathbb{S}(d)_{d+w}. 
    \end{align}
The category $\mathbb{S}(d)_w$ is called the \textit{quasi-BPS category} of $\mathbb{C}^3$ for $(d,w)\in \mathbb{N}\times\mathbb{Z}$ in~\cite{PT0}, see \cite{PT1} for its properties and computations of its K-theory. 


\subsection{More on weights}

\subsubsection{}\label{prime} 

Let $A$ be a partition $(d_i, w_i)_{i=1}^k$ of $(d, w)$ and consider its corresponding antidominant cocharacter $\lambda$. Define the weights
\begin{align*}
    \chi_A:=\sum_{i=1}^k w_i\tau_{d_i},\
    \chi'_A:=\chi_A+\mathfrak{g}^{\lambda>0}.
\end{align*}
Consider weights $\chi'_i\in M(d_i)_{\mathbb{R}}$ such that
\[\chi'_A=\sum_{i=1}^k \chi'_i.\]
 Let $v_i$ be the sum of coefficients of $\chi'_i$ for $1\leq i\leq k$; alternatively, $v_i:=\langle 1_{d_i}, \chi'_i\rangle$. We denote the above transformation by
  \begin{align}\label{trans:A}
     A\mapsto A', \ 
     (d_i, w_i)_{i=1}^k\mapsto (d_i, v_i)_{i=1}^k.
 \end{align}
Explicitly, the weights $v_i$ for $1\leq i\leq k$ are given by 
\begin{align}\label{w:prime}
    v_i=w_i+d_i\left(\sum_{j>i}d_j -\sum_{j<i}d_j \right). 
\end{align}
This transformation explains how to change weights under the Koszul equivalence \cite[Proposition 3.1]{P2}.

\subsubsection{}\label{dectree2}

We recall a construction from \cite{PT0}, see \cite[Subsections 3.3.2 and 3.3.3 and Propositions 3.4 and 3.7]{PT0} which is also used in the proof of Proposition \ref{prop:sodD}. Let $\mu\in \mathbb{R}$, and let $\chi$ be a dominant weight such that $\chi+\rho+\mu\sigma_d\in \mathbf{V}(d)$. Then there exists a unique partition $(d_i)_{i=1}^k$ of $d$, integers $(w_i)_{i=1}^k$ such that for $v_i$ as in \eqref{w:prime}, we have
\begin{equation}\label{ineq}
    -1-\mu\leq \frac{v_1}{d_1}<\cdots<\frac{v_k}{d_k}\leq -\mu,
\end{equation}
and weights $\chi_i \in M(d_i)$ such that
\[\chi=\sum_{i=1}^k\chi_i\text{ with }\chi_i+\rho_i+\mu\sigma_{d_i}\in \mathbf{W}(d_i)_{w_i}\text{ for }1\leq i\leq k.\]
Here, we denote by $\rho_i$ half the sum of positive roots of $GL(d_i)$ for $1\leq i\leq k$.

The converse is also true, and it follows from \cite[Proposition 3.8]{PT0}. Consider a partition $(d_i)_{i=1}^k$ of $d$, integers $(w_i)_{i=1}^k$ such that for $v_i$ as in \eqref{w:prime}, the inequality \eqref{ineq} holds, and weights $\chi_i$ with $\chi_i+\rho_i+\mu\sigma_{d_i}\in \mathbf{W}(d_i)_{w_i}$ for $1\leq i\leq k$. Let $\chi:=\sum_{i=1}^k \chi_i$. Then we have 
\[\chi+\rho+\mu\sigma_d\in\mathbf{V}(d).\]

\section{The categorical wall-crossing formula for DT/PT quivers} 
\label{sec:DT/PTlocal}

Fix $a\geq 1$. Let $\chi_0 \colon GL(V) \to \mathbb{C}^{\ast}$ be the determinant character
$g \mapsto \det g$. Recall the varieties: 
\begin{align*}
    I^a(d)&:=R^{af}(1, d)^{\chi_0\text{-ss}}/GL(d),\\
    P^a(d)&:=R^{af}(1, d)^{\chi^{-1}_0\text{-ss}}/GL(d).
\end{align*}
\begin{remark}\label{rmk:moduli}
It is easy to describe the $\chi_0^{\pm}$-semistable loci on $R^{af}(1, d)$: 
a $\chi_0$-semistable representation
(resp.~$\chi_0^{-1}$-semistable representation)
consists of tuples
\begin{align*}
(u_1, \ldots, u_{a+1}, v_1, \ldots, v_a, A, B, C) \in V^{\oplus (a+1)} \oplus (V^{\vee})^{\oplus a}
\oplus \mathfrak{gl}(V)^{\oplus 3}
\end{align*}
such that (see~\cite[Lemma~7.10]{Toddbir})
\begin{align*}
    \mathbb{C}\langle A, B, C \rangle (u_1, \ldots, u_{a+1})=V, 
    \ (\mathrm{resp}.~\mathbb{C}\langle A, B, C \rangle(v_1, \ldots, v_a)=V^{\vee}).
\end{align*}
\end{remark}

The main result of this section is Theorem \ref{thm2bis}.
The most important ingredient in its proof is the following semiorthogonal decomposition for 
subcategories of $D^b(\X^a(1, d))$:

\begin{thm}\label{thm2}
    Let $\mu\in \mathbb{R}$ and
    consider the weight $\delta:=\mu\sigma_d\in M(d)_{\mathbb{R}}$. 
    There is a semiorthogonal decomposition
    \begin{equation}\label{SOD32}
    \mathbb{D}^a(1, d; \delta)=\Big\langle \left(\boxtimes_{i=1}^k \mathbb{M}(d_i)_{w_i}\right)\boxtimes \mathbb{M}^a\left(1, d'; \delta'\right) \Big\rangle.
     \end{equation}
    The right hand side is after all $d'\leq d$, all decompositions $\sum_{i=1}^{k} d_i=d-d'$, and all integers $w_i\in \mathbb{Z}$ for $1\leq i\leq k$ such that for
    \begin{equation}\label{transvw}
        v_i:=w_i+d_i\left(d'+\sum_{j>i} d_j-\sum_{j<i}d_j\right),
    \end{equation} we have
    \begin{equation}\label{boundsalpha2}
    -1-\mu-\frac{a}{2}\leq \frac{v_1}{d_1}<\cdots<\frac{v_k}{d_k}\leq -\mu-\frac{a}{2},
\end{equation}
    and where $\delta':=\left(\mu-d+d'\right)\sigma_{d'}$.
    The orthogonality of the categories is as in Subsections \ref{comparison}.
    The functors \[\left(\boxtimes_{i=1}^k \mathbb{M}(d_i)_{w_i}\right)\boxtimes \mathbb{M}^a\left(1, d'; \delta'\right)\to \mathbb{D}^a(1, d; \delta)\] are given by the Hall product \eqref{prel:halla}.
\end{thm}
Given $\mu\in\mathbb{R}$ and a partition $(d_i)_{i=1}^k$ of $d'\leq d$, we refer to \eqref{boundsalpha2} as a condition on the tuple of integers $(w_i)_{i=1}^k$.


\subsection{Weight decompositions}

We discuss several preliminary results about decompositions of weights, or, alternatively, about decompositions of polytopes in $M(d)_{\mathbb{R}}$. 

\begin{prop}\label{prop35PT0}
Let $\chi$ be a strictly dominant weight such that $\chi\in \mathbf{V}^a(1, d)$. 
    There exists a decomposition
\begin{equation}\label{decomm}
D: \chi=\sum_{1\leq j<i\leq d} c_{ij}(\beta_i-\beta_j)+\sum_{i=1}^d c_i\beta_i
\end{equation}
with coefficients
\begin{equation}\label{coeffdec}
0\leq c_{ij}\leq \frac{3}{2} \text{ and } -\frac{a+2}{2}\leq c_i\leq \frac{a}{2} \text{ for }1\leq j<i\leq d. 
\end{equation}
\end{prop}

\begin{proof}
    The argument used in \cite[Proposition 3.5]{PT0} also applies here.
\end{proof}

\begin{prop}\label{prop34}
 	Let $\chi$ be a strictly dominant weight such that $\chi\in \mathbf{V}^a(1, d)$. Then there exists a unique $e\leq d$ such that 
 	\begin{equation}\label{chisumstrong}
 		\chi=\sum_{j<i\leq e}c_{ij}(\beta_i-\beta_{j})+\sum_{e<j<i}c_{ij}(\beta_i-\beta_j)+\sum_{j\leq e<i}\frac{3}{2}(\beta_i-\beta_j)+\sum_{1\leq i\leq d}c_i\beta_i,
 	\end{equation}
 	where $c_{ij}$ and $c_i$ satisfy
 	\begin{align*}
 		0\leq c_{ij} \leq \frac{3}{2} \mbox{ for } j<i, \ 
 		-\frac{a+2}{2} \leq c_i \leq -\frac{a}{2} \mbox{ if }i \leq e, \ 
 		-\frac{a}{2}<c_i \leq \frac{a}{2} \mbox{ if } i>e. 
 		\end{align*} 	
 	Let $f:=d-e$.
 	Alternatively, we have that
 	\begin{equation}\label{chisum}
 		\chi\in \left(-\frac{a+3f}{2}\sigma_e+\mathbf{V}(e)\right)+\left(\frac{3e}{2}\sigma_f+\mathbf{W}^a(1,f)\right).
 	\end{equation}
 \end{prop}
 
\begin{step}
There exists such $e\leq d$.
\end{step}
 \begin{proof}
 	Consider a decomposition 
 	\begin{align}\label{decom:D2}
 		D\colon \chi=\sum_{1\leq j<i\leq d} c_{ij}(\beta_i-\beta_j)+\sum_{1\leq i\leq d} c_i\beta_i
 	\end{align}
 	as in Proposition \ref{prop35PT0}.
 	A choice of the decomposition $D$ corresponds to a point in the fiber of the 
 	following continuous map at $\chi \in M(d)$:
 	\begin{align*}
 		\gamma \colon 
 		\left[0, \frac{3}{2}  \right]^{\times d(d-1)/2} \times \left[ -\frac{a+2}{2}, \frac{a}{2} \right]^{\times d}
 		\to M(d)_{\mathbb{R}}
 	\end{align*}
 	defined by sending $(c_{ij}, c_i)$ to the right hand side in (\ref{decom:D2}).  
For each decomposition $D$, we set 
\begin{align}\label{def:IJ}
	I:=\left\{i \relmiddle| c_i >-\frac{a}{2}\right\}, \ J:=\left\{i \relmiddle| c_i \leq -\frac{a}{2}\right\}. 
	\end{align}
 	Define the function $\sigma\colon \gamma^{-1}(\chi)\to\mathbb{R}$:
 	\[\sigma(D):=\sum_{i\in I}\left(c_i+\frac{a}{2}\right).\]
 	The function $\sigma$ is continuous and the set $\gamma^{-1}(\chi)$ is compact. 
 	Among all decompositions $D$ for which $\sigma$ is minimal, choose one for which $|I|$ is as large as possible. Let $m:=\sigma(D)$.

 	We first show that 
 	\begin{align}\label{contra1}
 		c_{ij}=\frac{3}{2} \mbox{ if }j<i, (i, j) \in I \times J. 
 		\end{align}
 	Assume there exists $j<i$ such that $(i, j) \in I \times J$ and $c_{ij}<3/2$. 
 	We take $0<\varepsilon \ll 1$ and set
 	$c_i'=c_i-\varepsilon$, $c_j'=c_j+\varepsilon$, $c_{ij}'=c_{ij}+\varepsilon$. 
  	Consider the decomposition $D'$ with the same coefficients as $D$ except for $c'_{ij}$, $c'_i$, and $c'_j$. 
 	Then $D'$ is a decomposition of $\chi$ because 
 	\[c_{ij}(\beta_i-\beta_j)+c_i\beta_i+c_j\beta_j=(c_{ij}+\varepsilon)(\beta_i-\beta_j)+(c_i-\varepsilon)\beta_i+(c_j+\varepsilon)\beta_j.\] 
 	The decomposition $D'$ satisfies \eqref{coeffdec}. If $c'_j<-a/2$, then we may assume that $I'=I$ and then $\sigma(D')<\sigma(D)$. If $c'_j=-a/2$, then $I'=I\cup\{j\}$ and $\sigma(D')=\sigma(D)$, which contradicts the maximality of $|I|$ among all decompositions with $\sigma(D)=m$. Thus
 	(\ref{contra1}) holds. 
 	
 	Similarly, 
 	we show that 
 		\begin{align}\label{contra2}
 		c_{ij}=0 \mbox{ if }j<i, (i, j) \in J \times I. 
 	\end{align}
 	Assume there exists $i>j$ such that $(i, j) \in J \times I$ and $c_{ij}>0$. 
 	We take $0<\varepsilon \ll 1$ and set  
 $c_i'=c_i+\varepsilon$, $c_j'=c_j-\varepsilon$, 
 $c'_{ij}=c_{ij}-\varepsilon$. 
 	Consider the decomposition $D'$
 	with the same coefficients as $D$ except for $c'_{ij}$, $c'_i$, and $c'_j$. Then $D'$ is a decomposition of $\chi$ because 
 	\[c_{ij}(\beta_i-\beta_j)+c_i\beta_i+c_j\beta_j=(c_{ij}-\varepsilon)(\beta_i-\beta_j)+(c_i+\varepsilon)\beta_i+(c_j-\varepsilon)\beta_j\] and $D'$ satisfies \eqref{coeffdec}. If $c_i<-a/2$, then we can choose $\varepsilon$ such that $I'=I$ and then $\sigma(D')<\sigma(D)$. 
 	If $c_i=-a/2$, then $I'=I\cup\{i\}$ and $\sigma(D')=\sigma(D)$, which contradicts the maximality of $|I|$ among decompositions $D$ with $\sigma(D)=m$. 
 	Thus (\ref{contra2}) holds.


 	We denote $f=|I|$, $e=|J|=d-f$, and write \[\chi=\sum_{i=1}^d b_i\beta_i, \ 
 	b_i=\sum_{j<i}c_{ij}-\sum_{i<j}c_{ji}+c_i.\] 
 	Let $i'\in I$ and $j' \in J$. Then from (\ref{contra1}) and (\ref{contra2})
 	and noting that $c_{ij} \in [0, 3/2]$ for all $j<i$, we have
 	\begin{align}\label{eiprime}
 		b_{i'}&>-\frac{a}{2}+\frac{3}{2}|J\cap \{1,\ldots, i'-1\}|-\frac{3}{2}|I\cap\{i'+1,\ldots, d\}|, \\ \notag
  		b_{j'}&\leq -\frac{a}{2}+\frac{3}{2}|J\cap\{1,\ldots, j'-1\}|-\frac{3}{2}|I\cap \{j'+1,\ldots, d\}|.
 	\end{align}
 	We claim that \begin{equation}\label{ijprime}
 		J=\{1,\ldots, e\}\text{ and }I=\{e+1,\ldots, d\}.
 	\end{equation} 
 	Assume the claim in \eqref{ijprime} is false. Thus there exist $i'\in I$, $j'\in J$ with $i'+1=j'$. Then $b_{i'}<b_{j'}$ because $\chi$ is strictly dominant. Combining inequalities 
 	 \eqref{eiprime}, we obtain:
 	\begin{multline*}
 		-\frac{a}{2}+\frac{3}{2}|J\cap\{1,\ldots, j'-1\}|-\frac{3}{2}|I\cap \{j'+1,\ldots, d\}|\\> -\frac{a}{2}+\frac{3}{2}|J\cap \{1,\ldots, i'-1\}|-\frac{3}{2}|I\cap\{i'+1,\ldots, d\}|,\end{multline*}
 	and so 
 	\begin{equation}\label{IJpos}
 		|J\cap\{i',\ldots, j'-1\}|+|I\cap \{i'+1,\ldots, j'\}|>0.
 	\end{equation}
 	However,  $I\cap\{i'+1,\ldots, j'\}=I\cap\{j'\}=\emptyset$ and $J\cap \{i',\ldots, j'-1\}=J\cap\{i'\}=\emptyset$.
 	Thus \eqref{IJpos} is false and so the claim in \eqref{ijprime} is true. Consider the weights in $\mathbf{V}(e)$ and $\mathbf{W}^a(1, f)$, respectively:
 	\begin{align*}
 		\chi_1&=\sum_{j<i\leq e}c_{ij}(\beta_i-\beta_j)+\sum_{i\leq e}\left(c_i+\frac{a}{2}\right)\beta_i,
 		\\
 		\chi_2&=\sum_{e<j<i}c_{ij}(\beta_i-\beta_j)+\sum_{i\geq e+1}c_i\beta_i.
 	\end{align*}
 	Then 
 	\[\chi=\left(-\frac{a+3f}{2}\sigma_e+\chi_1\right)+
 	\left(\frac{3e}{2}\sigma_f+\chi_2\right).\]
 	The conclusion thus follows.
 \end{proof}

\begin{step}
The $e\leq d$ with the desired property is unique.
\end{step}
\begin{proof}
    Assume there exist $e'<e$ such that \eqref{chisum} is satisfied for $e$ and $e'$. Let $f=d-e$.
    Recall that $\chi=\sum_{i=1}^d b_i\beta_i$.
    The weight $\chi$ has a description \eqref{chisumstrong} for $e'$, so 
    \begin{align}\label{chisum1}
    \sum_{e'<i\leq e} b_{i}
&=\sum_{e'<i\leq e}\left(\sum_{e'<j <i}c_{ij}
-\sum_{i<j}c_{ji}+\frac{3}{2}e'+c_i\right) 
\\    
  \notag
  &=-\sum_{e' <i \leq e<j}c_{ji}+\sum_{e'<i\leq e}\left(\frac{3}{2}e'+c_i\right) \\
  &\notag >\left(-\frac{3}{2}f+\frac{3}{2}e'-\frac{1}{2}a\right)(e-e').
    \end{align}
    
    Next, $\chi$ has a description \eqref{chisumstrong} for $e$.
    We abuse notation and denote the coefficients of this description also by $c_{ij}$ and $c_i$.
    Similarly to (\ref{chisum1}), we have 
    \begin{align}\label{chisum2}
     \sum_{e'<i\leq e} b_{i}
&=\sum_{e'<i\leq e}\left(-\sum_{i<j\leq e}c_{ji}
+\sum_{j<i}c_{ij}-\frac{3}{2}f+c_i\right) 
\\    
  \notag
  &=-\sum_{j\leq e' <i \leq e}c_{ij}+\sum_{e'<i\leq e}\left(-\frac{3}{2}f+c_i\right) \\
  &\notag \leq \left(\frac{3}{2}e'-\frac{3}{2}f-\frac{1}{2}a\right)(e-e').
    \end{align}
    We obtain a contradiction by comparing \eqref{chisum1} and \eqref{chisum2}.
\end{proof}

The decomposition of $\mathbf{V}^a(1, 2)$
into chambers constructed in Proposition~\ref{prop34} is depicted in Figure~2. 

	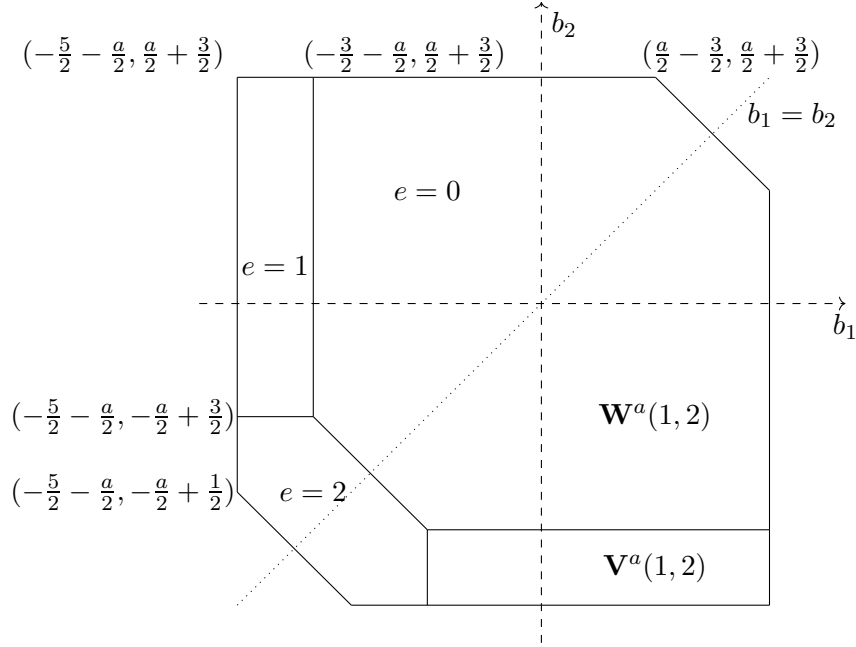
\begin{figure}[h]
	\begin{align*}
	\begin{tikzpicture}
		\draw (-3, 3) -- (1.5, 3);
		\draw (1.5, 3) -- (3, 1.5);
		\draw (3, 1.5) -- (3, -3);
		\draw (-3, 3) -- (-3, -1.5);
		\draw (-3, -1.5) -- (-1.5, -3);
		\draw (-1.5, -3) -- (3, -3);
		\draw (-4, 3) -- (-3, 3);	
		\draw (-4, 3) -- (-4, -2.5);		
		\draw (-4, -2.5) -- (-2.5, -4);
		\draw (-2.5, -4) -- (3, -4);
		\draw (3, -4) -- (3, -3);
		\draw (-3, -1.5) -- (-4, -1.5);
		\draw (-1.5, -3) -- (-1.5, -4);	
		\draw[->, dashed] (-4.5, 0) -- (4, 0);
		\draw[->, dashed] (0, -4.5) -- (0, 4);
		\draw (-1.5,1.5) node[circle] {$e=0$};
		\draw (-3.5,0.5) node[circle] {$e=1$};
		\draw (4,-0.3) node[circle] {$b_1$};
		\draw (0.3,3.7) node[circle] {$b_2$};
		\draw (3.3,2.5) node[circle] {$b_1=b_2$};		
		\draw (-3,-2.5) node[circle] {$e=2$};
		\draw (2.5,3.3) node[circle] {$(\frac{a}{2}-\frac{3}{2}, \frac{a}{2}+\frac{3}{2})$};
		\draw (-1.8,3.3) node[circle] {$(-\frac{3}{2}-\frac{a}{2}, \frac{a}{2}+\frac{3}{2})$};
		\draw (-5.5,3.3) node[circle] {$(-\frac{5}{2}-\frac{a}{2}, \frac{a}{2}+\frac{3}{2})$};
		\draw (-5.5,-1.5) node[circle] {$(-\frac{5}{2}-\frac{a}{2}, -\frac{a}{2}+\frac{3}{2})$};
		\draw (-5.5,-2.5) node[circle] {$(-\frac{5}{2}-\frac{a}{2}, -\frac{a}{2}+\frac{1}{2})$};
		\draw (1.5,-1.5) node[circle] {$\mathbf{W}^{a}(1, 2)$};
		\draw (1.5,-3.5) node[circle] {$\mathbf{V}^{a}(1, 2)$};
		\draw[dotted] (-4, -4) -- (3, 3);		
	\end{tikzpicture}
\end{align*}
	\caption{Decomposition of $\mathbf{V}^a(1, 2)$}
	\end{figure}

\subsection{Invariants of weights}

  \subsubsection{} 
Let $e\leq d$ be two natural numbers and let $f=d-e$. Consider the antidominant cocharacter \[\tau_e:=(\overbrace{t, \ldots, t}^{e}, \overbrace{1, \ldots, 1}^{f}): \mathbb{C}^*\to T(d).\] 
Fix $\mu\in \mathbb{R}$ and let $\delta:=\mu \sigma_d$.
    Let $\chi$ be a dominant weight such that $\chi+\rho+\delta\in \mathbf{V}^a(1, d)$. Let $e=e(\chi)$ be the integer $0\leq e\leq d$ such that
    \begin{align}\label{write:chi}
               \chi+\rho+\delta=\sum_{j<i\leq e}c_{ij}(\beta_i-\beta_{j})+\sum_{e<j<i}c_{ij}(\beta_i-\beta_j)+\sum_{j\leq e<i}\frac{3}{2}(\beta_i-\beta_j)+\sum_{1\leq i\leq d}c_i\beta_i,
            \end{align}
    where $c_{ij}$ and $c_i$ are as in 
    Proposition~\ref{prop34}. 
  We define $p(\chi)$ to be 
  \begin{align*}
        p(\chi)&:=\langle \tau_e, \chi+\rho+\delta\rangle+\frac{3}{2}ef+\frac{a}{2}e
        =\sum_{i=1}^e \left(c_i+\frac{a}{2} \right) 
        \leq 0. 
         \end{align*}
         The functions $e(\chi), p(\chi)$ depend on $\mu$.

 \subsubsection{}\label{typeweight}  
 
The following is a consequence of Proposition \ref{prop34} and the discussion in Subsection \ref{dectree2}.

\begin{prop}\label{prop:chi}
    Let $\mu\in \mathbb{R}$, let $\delta=\mu\sigma_d$, and let $\chi$ be a dominant (integral) weight such that $\chi+\rho+\delta\in \mathbf{V}^a(1, d)$. Then there exists a unique $d'\leq d$, partition $(d_i)_{i=1}^k$ of $e:=d-d'$, integers $(w_i)_{i=1}^k$ satisfying \eqref{boundsalpha2}, and weights $\chi_i\in M(d_i)$, $\chi'\in M(d')$
    such that $\chi=\chi_1+\cdots+\chi_k+\chi'$
    and 
     \begin{align*}
        \chi_i+\rho_i&\in \mathbf{W}(d_i)_{w_i}, \ 
        \chi'+\rho'+(\mu-e)\sigma_{d'}\in \mathbf{W}^a(1, d'). 
    \end{align*}
       In the above, $\rho_i$ and $\rho'$ are
       half the sums of positive roots of $GL(d_i)$ and
       $GL(d')$, respectively, where $1\leq i\leq k$. 
\end{prop}
Let $\mu\in \mathbb{R}$, let $\delta=\mu\sigma_d$, and let $\chi$ be a dominant weight such that $\chi+\rho+\delta\in \mathbf{V}^a(1, d)$. Let $S(\chi)=S$ be the ordered collection of pairs $S=(d_i, w_i)_{i=1}^k$
satisfying the condition in Proposition~\ref{prop:chi}. We say $S$ is the \textit{type} of $\chi$.
Then $p(\chi)$ is alternatively written as 
\begin{align}\label{pSchi}
p(\chi)&=\sum_{1\leq i\leq k}w_i-(d-d')\left(-\mu-d'-\frac{a}{2}\right) \\
\notag &=\sum_{1\leq i\leq k}v_i-(d-d')\left(-\mu-\frac{a}{2}\right)\leq 0.
\end{align}
We also denote the above number by $p(S)$.

\subsubsection{Comparison of ordered pairs}\label{comparison} 
Fix $d\in \mathbb{N}$ and $\mu\in \mathbb{R}$.
Let $R$ be the set \[R=\{(d_i, w_i)_{i=1}^k\mid d_i\in \mathbb{N}, w_i \in \mathbb{Z}\}.\] 
We define a subset $O\subset R\times R$ depending on $\mu$. A pair of types $(S', S)$ is in $O$ if either
\begin{itemize}
    \item $p(S')>p(S)$, or
    \item $p(S')=p(S)$ and $\sum_{i=1}^{k'} d'_i<\sum_{i=1}^{k}d_i$, or
    \item $p(S')=p(S)$, and $\sum_{i=1}^{k'} d'_i=\sum_{i=1}^{k}d_i$, and 
    $(S', S)$ is in the set $O$ from \cite[Subsection 3.4]{PT0}. 
\end{itemize}
If $(S', S)$ is in $O$, we also write $S>S'$.

For $0<\varepsilon\ll 1$, let $\mu=-a/2-\varepsilon$.
Recall the definition \eqref{transvw} of the integers $v_i$ from $w_i$. Then \[p(S)=\sum_{i=1}^k v_i-\varepsilon \sum_{i=1}^k d_i.\] In this case, the set $O$ contains pairs of types $(S', S)$ such that
\begin{itemize}
	    \item $\sum_{i=1}^{k'} v'_i >\sum_{i=1}^{k} v_i$, or 
	    \item $\sum_{i=1}^{k'} v'_i =\sum_{i=1}^{k} v_i$
	    and $\sum_{i=1}^{k'} d'_i<\sum_{i=1}^{k} d_i$, or 
	    \item $\sum_{i=1}^{k'} v'_i =\sum_{i=1}^{k} v_i$,
	    and $\sum_{i=1}^{k'} d'_i=\sum_{i=1}^{k} d_i$, and $(S', S)$
	    is in the set $O$ from~\cite[Subsection~3.4]{PT0}. 	\end{itemize}

    \subsection{Weight comparison}
    We discuss some preliminary results on weights. 
       We fix $\mu\in \mathbb{R}$ and let $\delta:=\mu\sigma_d\in M(d)_\mathbb{R}$.
    	For $\chi$ a weight and for $1\leq l\leq d$, define \begin{align*}
            p_l(\chi) :=\langle \tau_l, \chi+\rho+\delta\rangle+\frac{3}{2}l(d-l)+\frac{a}{2}l. 
                    \end{align*}	
		
    \begin{lemma}\label{lem:pl}
        Let $\chi$ be a dominant weight such that $\chi+\rho+\delta \in \mathbf{V}^a(1, d)$ and let $e=e(\chi)$. 
         Then $p_l(\chi) \geq p(\chi)$ and the inequality is strict if $l>e$. 
            \end{lemma}
            \begin{proof}
  	Note that $p_e(\chi)=p(\chi)$. 
  	We write 
	$\chi+\rho+\delta$
	as (\ref{write:chi}), 
	where $c_{ij}$ and $c_i$ are as in Proposition \ref{prop34}. We have
	\begin{align*}
		p_l(\chi) &=-\sum_{i\leq l <j}c_{ji}
		+\sum_{i\leq l}c_i+\frac{3}{2}l(d-l)+\frac{a}{2}l \\
		&\geq \sum_{i\leq l}\left(c_i+\frac{a}{2}\right)
		\stackrel{(\star)}{\geq}  \sum_{i\leq e}\left(c_i+\frac{a}{2}\right)=p(\chi).
		\end{align*}
	The inequality $(\star)$ holds from Lemma~\ref{lem:obvious}, 
	and is strict if $l>e$. 
	\end{proof}
	We have used the following lemma, whose proof is 
	obvious: 
\begin{lemma}\label{lem:obvious}
	Let $x_1, \ldots, x_d \in \mathbb{R}$ and let
	$S:=\{i \mid x_i \leq 0\}\subset \{1,\ldots, d\}$. 
	Then for any subset $T \subset \{1, \ldots, d\}$, we have 
	\begin{align*}
		\sum_{i\in T}x_i \geq \sum_{i\in S}x_i.
		\end{align*}
	If equality holds, then $T \subset S$. 
If equality holds and $\lvert T \rvert =\lvert S \rvert$, 
	then $T=S$. 
		\end{lemma}	

For $e\leq d$, consider the multiset of weights $\mathcal{W}^a_e:=
    R^a(1, d)^{\tau_e<0}$, or explicitly 
    \begin{align*}
      \mathcal{W}^a_e=\{(\beta_i-\beta_j)^{\times 3}, (-\beta_j)^{\times a}\mid j\leq e<i\}. 
    \end{align*}

  \begin{prop}\label{prop37}
    Let $\chi$, $\chi'$ be dominant weights such that 
    $\chi+\rho+\delta \in \mathbf{V}^a(1, d)$ and $\chi'+\rho+\delta \in \mathbf{V}^a(1, d)$. 
    Let $e=e(\chi)$, $e'=e(\chi')$, and let $I$ be a subset $\mathcal{W}_{e'}^a$. 
    Let $e''=e((\chi'-\sigma_I)^+)$. 
    Suppose that either 
    \begin{enumerate}[(i)]
        \item $p(\chi')>p(\chi)$ or
        \item $p(\chi')=p(\chi), e'<e$ or \item $p(\chi')=p(\chi)$, $e'=e$, and $I\neq \emptyset$.
    \end{enumerate}
    Then we have 
\begin{align}\label{ineq:pl}
    p_e((\chi'-\sigma_I)^+) \geq p((\chi'-\sigma_I)^+) \geq p(\chi). 
\end{align}
Moreover, if the second inequality is an equality, then 
$p(\chi')=p(\chi)$ and $e'' \leq e'$, 
and $e'' < e'$ if $I\neq \emptyset$. 
Hence $e''<e$, and thus the 
first inequality is strict by Lemma~\ref{lem:pl}. 
In particular, in each of the above three cases, we have
\begin{align*}
\langle \tau_e, (\chi'-\sigma_I)^+ \rangle> \langle \tau_e, \chi \rangle. 
\end{align*}
     \end{prop}
\begin{proof}
	As the first inequality in (\ref{ineq:pl}) follows from Lemma~\ref{lem:pl}, 
	it is enough to prove the second inequality. 
	Write 
	\begin{align*}
				\chi'+\rho+\delta&=\sum_{1\leq i,j\leq d}c'_{ij}(\beta_i-\beta_j)+\sum_{1\leq i\leq d}c'_i\beta_i,\\
		\chi'-\sigma_I+\rho+\delta&=\sum_{1\leq i,j\leq d}\widetilde{c}_{ij}(\beta_i-\beta_j)+\sum_{1\leq i\leq d}\widetilde{c}_i\beta_i,
	\end{align*}
	where $c'_{ij}, c'_i$ are as in Proposition \ref{prop34}. Then
	$|\widetilde{c}_{ij}|\leq 3/2$ for all $1\leq i, j\leq d$, and $\widetilde{c}_i\geq c'_i$ for $1\leq i\leq d$. 
	Let $w\in \mathfrak{S}_d$ be such that $(\chi'-\sigma_I)^++\rho=w(\chi'-\sigma_I+\rho)$.
	Then using Lemma~\ref{lem:obvious} and the 
	assumption (i) or (ii) or (iii), we have 
	\begin{align}\label{ineqs}
		p((\chi'-\sigma_I)^+)&=\sum_{w(i) \leq e''}\left(\widetilde{c}_i+\frac{a}{2}\right) \stackrel{(\star)}{\geq} 
		\sum_{i=1}^{e''}\left(\widetilde{c}_i+\frac{a}{2}\right)
		 \geq
		\sum_{i=1}^{e''}\left(c_i'+\frac{a}{2}\right) \\
		&\notag \stackrel{(\diamondsuit)}{\geq}   
		\sum_{i=1}^{e'}\left(c_i'+\frac{a}{2}\right)
		=p(\chi') \geq p(\chi). 
				\end{align}	
	Therefore the second inequality in (\ref{ineq:pl}) also holds. 
	Suppose that every inequality in (\ref{ineqs})
	is an equality. Then $p(\chi)=p(\chi')$ and
	 $e'' \leq e'$ from the equality of $(\diamondsuit)$ and 
	 Lemma~\ref{lem:obvious}. 
	 Assume that $e''=e'$ and $I\neq \emptyset$. 
	From the equality of $(\star)$ and $e''=e'$, by Lemma~\ref{lem:obvious}
	we have 
$\{w(1), \ldots, w(e')\}=\{1, \ldots, e'\}$. 
Then 
we have 
\begin{align*}
	\langle \tau_{e'}, (\chi'-\sigma_I)^+ \rangle
=\langle \tau_{e'}, \chi'-\sigma_I \rangle
=\langle \tau_{e'}, \chi' \rangle-\langle \tau_{e'}, \sigma_I \rangle.
\end{align*}
As $I \subset \mathcal{W}_{e'}^a$ is non-empty, we have 
$\langle \tau_{e'}, \sigma_I \rangle<0$. 
Then $p((\chi'-\sigma_I)^+)>p(\chi')$, which contradicts that 
(\ref{ineqs}) are equalities. 
\end{proof}

 \subsection{Generation}

Fix $\mu\in\mathbb{R}$ and let $\delta:=\mu\sigma_d\in M(d)_\mathbb{R}$.

\begin{prop}\label{generationprop}
    The category $\mathbb{D}^a(1, d; \delta)$ is generated by the images of the Hall products \begin{equation}\label{Hallproduct3}
    \left(\boxtimes_{i=1}^k \mathbb{M}(d_i)_{w_i}\right)\boxtimes \mathbb{M}^a\left(1, d'; \delta'\right)\to \mathbb{D}^a(1, d; \delta)
    \end{equation}
    for all $d'\leq d$, all decompositions $(d_i)_{i=1}^k$ of $d-d'$, and tuples of integers $(w_i)_{i=1}^k$ satisfying \eqref{boundsalpha2},
    and where $\delta':=\left(\mu-d+d'\right)\sigma_{d'}$.
\end{prop}

We first explain that the images of the functors \eqref{Hallproduct3} are indeed in the category $\mathbb{D}^a(1, d; \delta)$.

\begin{prop}\label{prop:inside}
    Let $(d_i, w_i)_{i=1}^k$ and $d'$ be as above. Then the image of the Hall product 
    \[\left(\boxtimes_{i=1}^k \mathbb{M}(d_i)_{w_i}\right)\boxtimes \mathbb{M}^a\left(1, d'; \delta'\right)\to D^b(\X^a(1, d))\] is inside $\mathbb{D}^a(1, d; \delta)$. 
\end{prop}

\begin{proof}
    By Proposition~\ref{prop:sodD}, it suffices to show that the image of the functor
    \[p_{e, d'*}q_{e,d'}^*\colon \mathbb{D}(e; \delta_e)\boxtimes \mathbb{M}^a(1, d'; \delta')\to D^b(\X^a(1, d))\] 
       lies in $\mathbb{D}^{a}(1, d; \delta)$, where $\delta_e:=\left(\mu+d'+a/2\right)\sigma_e$. Consider weights $\chi_e\in M(e)$ and $\chi_{d'}\in M(d')$ such that 
    $\chi_e+\rho_e+\delta_e\in \mathbf{V}(e)$ and $\chi_{d'}+\rho_{d'}+\delta'\in \mathbf{W}^a(1, d')$. 
    It suffices to check that 
    \begin{align}\label{Gamma:in}p_{e,d'*}q_{e,d'}^*\left(\mathcal{O}_{\mathcal{Y}}\otimes \Gamma_L(\chi_e+\chi_{d'})\right) \in \mathbb{D}^{a}(1, d; \delta),
    \end{align}
    where $L:=GL(e)\times GL(d')$ and $\mathcal{Y}:=\X(e)\times\X(d')$. Let $\chi:=\chi_e+\chi_{d'}$. 
    We have 
    \begin{align*}
        \chi+\rho+\delta&=(\chi_e+\rho_e+\delta_e)+(\chi_{d'}+\rho_{d'}+\delta')
        +(\rho-\rho_e-\rho_{d'})+(\delta-\delta_e-\delta') \\
        & \in \mathbf{V}(e)+\mathbf{W}^a(1, d')-\left(\frac{3}{2}d'+\frac{a}{2}\right)
        \sigma_e+\frac{3}{2}e \sigma_{d'}.
    \end{align*}
    Then we can write 
      \[\chi+\rho+\delta=\sum_{i,j\leq d}c_{ij}(\beta_i-\beta_j)+\sum_{i\leq d}c_i\beta_i\]
      where $c_{ij}$ and $c_i$ satisfy 
      $|c_{ij}|\leq 3/2$ with $c_{ij}=3/2$ and $c_{ji}=0$ for $j\leq e<i$, $-1\leq c_i+a/2\leq 0$ for $1\leq i\leq e$, $|c_i|\leq a/2$ for $e<i\leq d$. 
    We take $I\subset \mathcal{W}^a_e$
    and write $\sigma_I$
    as 
    \begin{align*}
     \sigma_I=\sum_{j\leq e<i} f_{ij}(\beta_i-\beta_j)-\sum_{i\leq e}f_i\beta_i
     \end{align*}
     for $0\leq f_{ij}\leq 3$ and $0\leq f_i\leq a$. Then $\chi-\sigma_I+\rho+\delta$ is in $\mathbf{V}^a(1, d)$. The polytope $\mathbf{V}^a(1, d)$ is Weyl invariant, so $(\chi-\sigma_I)^++\rho+\delta$ is in $\mathbf{V}^a(1, d)$.
    Thus we have (\ref{Gamma:in}) by Proposition \ref{bbw}.
\end{proof}

\begin{proof}[Proof of Proposition \ref{generationprop}]
    By Proposition~\ref{prop:sodD}, it suffices to check that the category $\mathbb{D}^a(1, d; \delta)$ is generated by the images of the Hall products \[p_{e, d'*}q_{e,d'}^*\colon \mathbb{D}(e; \delta_e)\boxtimes \mathbb{M}^a(1, d'; \delta')\to \mathbb{D}^a(1, d; \delta)\] for all $e\leq d$, $d'=d-e$, and $\delta_e:=\left(\mu+d'+a/2\right)\sigma_e$. Let $\chi$ be a dominant weight such that $\chi+\rho+\delta\in \mathbf{V}^a(1, d; \delta)$. We need to show that $\mathcal{O}_{\X(d)}\otimes \Gamma_{GL(d)}(\chi)$ is generated by the images of these functors. The function $p(\chi)$ takes finitely many values. We use induction on the pair $(-p(\chi), e(\chi))$ ordered lexicographically. 
    Use Proposition \ref{prop34} for $\chi+\rho+\delta$ and write
    $\chi=\chi_e+\chi_{d'}$ such that 
    \[\chi_e+\rho_e+\delta_e\in \mathbf{V}(e), \ 
    \chi_{d'}+\rho_{d'}+\delta' \in\mathbf{W}^a(1, d').\]
    By Proposition \ref{bbw}, the cone of the map 
    \[\mathcal{O}_{\X(d)}\otimes \Gamma_{GL(d)}(\chi)\to p_{e,d'*}q_{e,d'}^*\left(\mathcal{O}_{\mathcal{Y}}\otimes \Gamma_L(\chi_e+\chi_{d'})\right)\] is generated by the vector bundles $\mathcal{O}_{\X(d)}\otimes \Gamma_{GL(d)}((\chi-\sigma_I)^+)$ for non-empty multisets $I\subset \mathcal{W}^a_e$. 
    
    If $(-p(\chi), e(\chi))=(0, 0)$, then $\chi+\rho+\delta\in\mathbf{W}^a(1, d;\delta)$. 
    Assume $(-p(\chi), e(\chi))$ is different from $(0, 0)$ and let $I$ be a non-empty subset of $\mathcal{W}^a_e$.
    Then as in the proof of Proposition~\ref{prop:inside}, we have 
     $(\chi-\sigma_I)^++\rho+\delta \in \mathbf{V}^a(1, d)$. 
    By Proposition \ref{prop37}, either $-p((\chi-\sigma_I)^+)<-p(\chi)$ or $p((\chi-\sigma_I)^+)=p(\chi)$ and $e((\chi-\sigma_I)^+)<e(\chi)$. The claim thus follows by the induction of $(-p(\chi), e(\chi))$.
\end{proof}

\subsection{Orthogonality}

In this subsection, we discuss orthogonality of the categories appearing on the right hand side of \eqref{SOD32}. 
	   For $S=(d_i, w_i)_{i=1}^k$, let 
	   $e=\sum_{i=1}^k d_i$ and 
	   $f:=d-e$.
	   We denote by $\lambda_S$ the antidominant cocharacter
    \begin{equation}\label{cocharacterS}
    \lambda_S:=(\overbrace{t^{k}, \ldots, t^{k}}^{d_1},\ldots, \overbrace{t, \ldots, t}^{d_k},\overbrace{1, \ldots, 1}^{f}):\mathbb{C}^*\to T(d).
        \end{equation}
        In order to simplify the notation, 
        we set 
        $\mathcal{X}:=\mathcal{X}^a(1, d)$, 
        $\mathcal{Y}:=\mathcal{X}^a(1, d)^{\lambda_S \geq 0}$
        and $\mathcal{Z}:=\mathcal{X}^a(1, d)^{\lambda_S}$.
      We have the maps from the 
    attracting stacks for the antidominant cocharacter $\lambda_S$, 
    see (\ref{attract:a})
    \begin{align*}
        \mathcal{Z} \stackrel{q}{\leftarrow} \mathcal{Y}
        \stackrel{p}{\to} \mathcal{X}. 
    \end{align*}
For another $S'=(d_i', w_i')$, we define $(e', f', \lambda_{S'}, \mathcal{X}', \mathcal{Y}', \mathcal{Z}', 
p', q')$ in 
a similar way.

\begin{prop}\label{prop311}
Suppose that $(S', S)\in O$, see Subsection \ref{comparison}.
Let $\mu\in\mathbb{R}$, and 
set $\delta_f=(\mu-e)\sigma_f \in M(f)_{\mathbb{R}}$ 
and $\delta_{f'}=(\mu-e')\sigma_{f'} \in M(f')_{\mathbb{R}}$.
Let $A, \widetilde{A}$ and $B$ be objects 
\begin{align*}
    A, \widetilde{A} \in \left(\boxtimes_{i=1}^k \mathbb{M}(d_i)_{w_i}\right)\boxtimes \mathbb{M}^a\left(1, f; \delta_f\right), \ 
    B \in \left(\boxtimes_{i=1}^{k'} \mathbb{M}(d'_i)_{w'_i}\right)\boxtimes \mathbb{M}^a\left(1, f'; \delta_{f'}\right). 
\end{align*}
We have 
\begin{align*}
    \mathrm{Hom}_{\X}(p'_*q'^*B, p_*q^*A)=0, \
    \mathrm{Hom}_{\mathcal{Z}}(\widetilde{A}, A)\stackrel{\cong}{\to}
    \mathrm{Hom}_{\X}(p_*q^*\widetilde{A}, p_*q^*A).
\end{align*}

\end{prop}

\begin{proof}
    We discuss the first equality. Let $\chi$ and $\chi'$ be dominant weights
    such that 
    \begin{align*}
        \chi+\rho+\delta \in \mathbf{V}^a(1, d), \ 
        \chi'+\rho+\delta \in \mathbf{V}^a(1, d),  
    \end{align*}
    which are of types $S$, $S'$ respectively. 
         Consider the Levi group $L:=GL(d)^{\lambda_S}$, 
         and define
         $L'$ analogously for $S'$.
    We may assume that $A=\mathcal{O}_{\mathcal{Z}}\otimes \Gamma_L(\chi)$ and $B=\mathcal{O}_{\mathcal{Z}'}
    \otimes\Gamma_{L'}(\chi')$.
    By adjunction, it suffices to check that 
    \[\text{Hom}_{\mathcal{Y}}(p^*p'_*q'^*B, q^*A)=0.\]
    Assume first that $e=e'$ and $\sum_{i=1}^k w_i=\sum_{i=1}^{k'}w'_i$. Then the conclusion follows as in \cite[Proposition 3.9]{PT0}. More explicitly, by \cite[Proposition 4.2]{P} and Proposition \ref{bbw}, it suffices to check that the $\lambda_S$-weights of the vector bundles in the resolution of $p^*p'_*q'^*B$ are strictly greater than the $\lambda_S$-weight of $\chi$, and this is verified in \cite[Proposition 4.3]{P}.
    
    Assume next that either $p(S')>p(S)$ or $p(S)=p(S')$ and $e>e'$. By \cite[Proposition 4.2]{P} and Proposition \ref{bbw}, it suffices to check that
    \begin{align}\label{ineq:taue}
    \langle \tau_e, (\chi'-\sigma_I)^+\rangle>\langle \tau_e, \chi\rangle
    \end{align}
    for all $I\subset \mathcal{W}^a_{S'}$, 
   where $\mathcal{W}^a_{S'}$ be the multiset of weights of $R^a(1, d)^{\lambda_{S'}<0}$. 
  Since $\mathcal{W}_{S'}^a \subset \mathcal{W}_e^a$, the inequality 
  (\ref{ineq:taue}) follows from Proposition \ref{prop37}. Note that \cite[Proposition 4.2]{P} can be applied ad litteram when comparing $\lambda_S$-weights, however the proof in loc.~cit.~applies to $\tau_e$-weights as well because $\tau_e$ fixes $R^a(d)^{\lambda_S}$ and acts with non-negative weights on $R^a(d)^{\lambda_{S}<0}$.

    We next discuss the second isomorphism. 
    By \cite[Proposition 3.9]{PT0}, it suffices to show that the functor
    \[p_{*}q^*\colon \mathbb{D}(e; \delta_e)\boxtimes \mathbb{M}^a(1, f; \delta_f)\to \mathbb{D}^a(1, d; \delta)\] is fully faithful, where $\delta_e:=\left(\mu+f+a/2\right)\sigma_e$. We may assume that $\widetilde{A}=\mathcal{O}_{\mathcal{Z}}\otimes \Gamma_L(\widetilde{\chi})$ for $\widetilde{\chi}$ a dominant weight of type $S$. By adjunction, it suffices to show that 
    \[\Hom_{\mathcal{Y}}(p^*p_*q^*\widetilde{A}, q^*A)\cong \Hom_{\mathcal{Y}}(q^*\widetilde{A}, q^*A)\cong \Hom_{\mathcal{Z}}(\widetilde{A}, A).\]
    Let $C$ be the cone of the morphism $p^*p_*q^*\widetilde{A}\to q^*\widetilde{A}$. It suffices to check that $\Hom_{\mathcal{Y}}(C, q^*A)=0$.
    By \cite[Proposition 4.2]{P} and Proposition \ref{bbw}, it suffices to check that \[\langle \tau_e, (\widetilde{\chi}-\sigma_I)^+\rangle>\langle \tau_e, \chi\rangle\] for all non-empty $I\subset \mathcal{W}^a_{S'}$. This follows from Proposition \ref{prop37}.
\end{proof}

\subsection{Proofs of the theorems}\label{ss:proof}
    In this subsection, we prove Theorems \ref{thm2} and \ref{thm2bis} and Corollary \ref{thm12}. 
    
\begin{proof}[Proof of Theorem \ref{thm2}]
    The statement follows from Propositions \ref{generationprop} and \ref{prop311}. 
\end{proof}    
    
Recall the subcategory $\mathbb{E}^a(1, d; \delta)$ of $D^b(\X^{af}(1, d))$ from Subsection \ref{defcategories}. Consider the projection map $b\colon \X^{af}(1, d)\to \X^a(1, d)$. 

\begin{prop}\label{prop312}
    There is a semiorthogonal decomposition
    \begin{equation}\label{SOD33}
    \mathbb{E}^a(1, d; \delta)=\Big\langle \left(\boxtimes_{i=1}^k \mathbb{M}(d_i)_{w_i}\right)\boxtimes \mathbb{F}^a(1, d'; \delta') \Big\rangle,
     \end{equation} where the right hand side is as in Theorem \ref{thm2}.
\end{prop}
    
    \begin{proof}
    By Theorem~\ref{thm2},
        the result follows as \cite[Corollary 3.11]{PT0} using \cite[Proposition 3.10]{PT0}. 
    \end{proof}

\begin{prop}\label{prop313}
    Let
    $\mu \in \mathbb{R}$
    with $2\mu l \notin \mathbb{Z}$
    for $1\leq l \leq d$
    and let $\delta:=\mu \sigma_d$.
    Consider the inclusions
    \[\iota\colon I^a(d)\hookrightarrow \X^{af}(1, d),\, \iota'\colon P^a(d)\hookrightarrow \X^{af}(1,d).\] Then the following functors are equivalences
    \begin{align*}
    \mathbb{E}^a(1, d; \delta)&\hookrightarrow D^b(\X^{af}(1, d))\xrightarrow{\iota^*} D^b(I^a(d)),\\
        \mathbb{F}^a(1, d; \delta)&\hookrightarrow D^b(\X^{af}(1, d))\xrightarrow{\iota'^*} D^b(P^a(d)).
    \end{align*}
\end{prop}    
    
\begin{proof}
    The same proof as in \cite[Proposition 3.13]{PT0} applies here. We refer to loc. cit. for full details. 
    The main tools used in the proof are window categories \cite{halp} and the magic window theorem for symmetric representations \cite{hls}.
    
    The Kempf-Ness loci for $I^a(d)$ and $P^a(d)$ are attracting loci for the cocharacters $\tau_e^{-1}$ and $\tau_e$, respectively, see \cite[Lemma 3.14]{PT0}. 
    We first explain fully faithfulness of the first functor.
    Let $I$ be the set of Kempf-Ness strata and consider the stratification \[\X^{af}(1, d)=I^a(d)\sqcup \bigsqcup_{i\in I}\mathcal{S}_i.\] 
    Let $\lambda_i=\tau_{e_i}^{-1}$ be the cocharacter corresponding to  $\mathcal{S}_i$, let $\mathcal{Z}_i:=\mathcal{S}_i^{\lambda_i}$, and let 
    \begin{align*}
    \eta_i :=\langle \lambda_i, (\mathcal{X}^{af}(1, d)^{\vee})^{\lambda_i>0} \rangle=
    -\langle \lambda_i, \mathcal{X}^{af}(1, d)^{\lambda_i<0} \rangle=(a+1)e_i+2e_i(d-e_i).  
\end{align*}
    Let $\mathbb{G}_{\eta}\subset D^b\left(\X^{af}(1, d)\right)$ be the subcategory of complexes $\mathcal{F}$ such that \begin{align}\label{wt:condG}
    \mathrm{wt}_{\lambda_i}(\mathcal{F}|_{\mathcal{Z}_i})
    +\left\langle \lambda_i, \delta+\frac{1}{2}d\tau_d \right\rangle \subset 
    \left[-\frac{1}{2}\eta_i, \frac{1}{2}\eta_i \right)
\end{align}
for all $i\in I$.
By \cite[Theorem 2.10]{halp}, the restriction $\iota^*$ induces an equivalence
\begin{equation}\label{iotaGv}
\iota^* \colon \mathbb{G}_{\eta}\xrightarrow{\sim} D^b(I^a(d)).
\end{equation}
On the other hand, 
as $\lambda_i=\tau_{e_i}^{-1}$,
we have that
$\X^{af}(1, d)^{\lambda_i<0}=\X^{a+1}(1, d)^{\lambda_i<0}$, and thus 
\begin{align}\label{vlambda}
    \eta_i=\langle \lambda_i, \X^{a+1}(1, d)^{\lambda_i>0} \rangle. 
\end{align}
By the assumption $2\mu l \notin \mathbb{Z}$
for $1\leq l\leq d$
and $\langle \lambda_i, \delta \rangle=-\mu e$, 
$\langle \lambda_i, d\tau_d\rangle=-e$, 
the condition (\ref{wt:condG}) is equivalent to 
\begin{align}\notag
    \mathrm{wt}_{\lambda_i}(\mathcal{F}|_{\mathcal{Z}_i})
    +\left\langle \lambda_i, \delta+\frac{1}{2}d\tau_d \right\rangle \subset 
    \left[-\frac{1}{2}\eta_i, \frac{1}{2}\eta_i \right]. 
\end{align}
Thus 
$\mathbb{E}(d; \delta)\subset \mathbb{G}_{\eta}$ by (\ref{vlambda}) and~\cite[Lemma~2.9]{hls}, noting 
that $R^{a+1}(1, d)$ is a symmetric $GL(d)$-representation. 
    Fully faithfulness of the second functor follows similarly.
    
We next discuss essential surjectivity. 
We have the projection maps 
\begin{align*}
\mathcal{X}^{a+1}(1, d) \stackrel{c}{\to} \mathcal{X}^{af}(1, d)
\stackrel{b}{\to} \mathcal{X}^a(1, d)
\end{align*}
where $c$ (resp.~$b$)
forgets 
one arrow from $1$ to $0$, (resp.~$0$ to $1$). 
The maps $c$, $b$ are 
affine bundles of relative dimension $d$. 
By Remark~\ref{rmk:moduli}, the $\chi_0$-stability (resp.~$\chi_0^{-1}$-stability) on $R^{af}(1, d)$
does not impose constraint on maps corresponding to edges from $1$ to $0$
(resp.~$0$ to $1$).
We have similar descriptions of $\chi_0$-semistable loci for $R^{a+1}(1,d)$ (resp.~$\chi^{-1}_0$-semistable loci for $R^a(1,d)$), see~\cite[Lemma~7.10]{Toddbir}. 
Therefore the projection maps $c$, $b$
restrict to 
affine bundles of relative dimension $d$
\begin{align*}
    c \colon \mathcal{X}^{a+1}(1, d)^{\chi_0 \text{-ss}} \to I^a(d), \ 
    b \colon P^a(d) \to \mathcal{X}^a(1, d)^{\chi_0^{-1}\text{-ss}}. 
\end{align*}
Then the essential surjectivity follows from the magic window theorem \cite[Theorem 3.2]{hls} for 
the symmetric stacks
$\mathcal{X}^{a+1}(a, d)$, 
$\mathcal{X}^a(1, d)$ as in \cite[Proof of Proposition 3.13]{PT0}. 
\end{proof}    
   
\begin{proof}[Proof of Theorem \ref{thm2bis}]
    The claim follows from Propositions \ref{prop312} and \ref{prop313}, noting 
    that the most left and right inequalities 
    in (\ref{boundsalpha2}) are never equalities 
    by the condition $2\mu l \notin \mathbb{Z}$
    for $1\leq l \leq d$. 
\end{proof}

\begin{proof}[Proof of Corollary \ref{thm12}]
    Taking the Grothendieck group of the categories in \eqref{SOD11}, we obtain the isomorphism
    \begin{equation}\label{decoK}
    K(I^a(d))\cong \bigoplus K\left((\boxtimes_{i=1}^k \mathbb{M}(d_i)_{w_i})\boxtimes P^a(d')\right),
     \end{equation}
    where the sum of the right hand side is as in Theorem \ref{thm2bis}. 
    Choose $d'\leq d$. By \cite[Theorem 1.1]{PT0}, there is a semiorthogonal decomposition 
    \begin{equation}\label{SODdprime}
    D^b(\mathrm{NHilb}(d-d'))=\Big\langle \boxtimes_{i=1}^k \mathbb{M}(d_i)_{w_i}\Big\rangle
    \end{equation}
    where the right hand side is after all partitions $(d_i)_{i=1}^k$ of $d-d'$ and integers $(w_i)_{i=1}^k$ such that for $v_i:=w_i+d_i\left(d'+\sum_{j>i} d_j-\sum_{j<i}d_j\right)$, the inequality \eqref{boundsalpha} is satisfied. There is thus a semiorthogonal decomposition 
    \[D^b\left(\mathrm{NHilb}(d-d')\times P^a(d')\right)=
    \Big\langle \left(\boxtimes_{i=1}^k \mathbb{M}(d_i)_{w_i}\right)\boxtimes D^b(P^a(d'))\Big\rangle,\] and thus a decomposition in K-theory:
    \[K\left(\mathrm{NHilb}(d-d')\times P^a(d')\right)\cong \bigoplus K\left(( \boxtimes_{i=1}^k \mathbb{M}(d_i)_{w_i})\boxtimes P^a(d')\right),\] where the right hand side is as for \eqref{SODdprime}. By \eqref{decoK}, we obtain 
    \begin{equation}\label{K}
        K(I^a(d))\cong \bigoplus_{d'=0}^d K\left(\mathrm{NHilb}(d-d')\times P^a(d')\right).
    \end{equation}
    We finally show that there is a Künneth isomorphism
    \begin{equation}\label{Kunneth}
        K\left(\mathrm{NHilb}(d-d')\times P^a(d')\right)\cong  K(\mathrm{NHilb}(d-d'))\otimes K(P^a(d')),
    \end{equation}
    which implies the conclusion by \eqref{K}. We use \cite[Theorem 2.10]{halp} to choose window subcategories $\mathbb{G}\subset D^b\left(\X^f(1, d-d')\right)$ and $\mathbb{H}\subset D^b\left(\mathcal{X}^{af}(1, d')\right)$ such that the restriction maps to the stable locus give equivalences 
    \[\mathbb{G}\xrightarrow{\sim}D^b(\mathrm{NHilb}(d-d')),\,
    \mathbb{H}\xrightarrow{\sim}D^b(P^a(d')).\] 
    Then, by loc. cit., the category $\mathbb{G}\boxtimes\, \mathbb{H}$ is part of a semiorthogonal decomposition of $D^b\left(\X^f(1, d-d')\times\mathcal{X}^{af}(1, d')\right)$ such that the restriction map gives an equivalence
    \[\mathbb{G}\boxtimes \mathbb{H}\xrightarrow{\sim}D^b\left( \mathrm{NHilb}(d-d')\times P^a(d')\right).\]
    The following diagram commutes and the spaces on the left column are summands of the spaces of the right column:
    \begin{equation*}
        \begin{tikzcd}
            K(\mathbb{G})\otimes K(\mathbb{H})\arrow[d, "a"]\arrow[r, hook]& K(\X^f(d-d'))\otimes K(\mathcal{X}^{af}(d'))
            \arrow[d, "b"]\arrow[l, two heads, shift left=0.13 cm]\\
            K(\mathbb{G}\boxtimes\,\mathbb{H})\arrow[r, hook]& K(\X^f(d-d')\times\mathcal{X}^{af}(d')) \arrow[l, two heads, shift left=0.13 cm].
        \end{tikzcd}
    \end{equation*}
    To show that $a$ is an isomorphism, it suffices to show that $b$ is an isomorphism, which is clear because  $K(\mathcal{X}^{af}(1, d'))\cong K_{GL(d')}(\text{pt})$,
    $K(\X^f(1, d-d'))\cong K_{GL(d-d')}(\text{pt})$,
    and $K(\X^f(1, d-d')\times\mathcal{X}^{af}(d'))\cong K_{GL(d-d')\times GL(d')}(\text{pt})$.
\end{proof}

\subsection{A variant of Theorem~\ref{thm2bis}}
We discuss a slight extension of Theorem~\ref{thm2bis} for quivers obtained from $Q^{af}$ by adding loops at the vertex $0$ and imposing a super-potential. We will use this extension in the proof of Theorem \ref{thm:exam}.

Let $Q^{af, N}$ be the quiver adding $N$-loops at the vertex $0$. 
Then the affine space of representations of $Q^{af, N}$ 
of dimension $(1, d)$ is 
\begin{align*}
    R^{af, N}(1, d)=\mathbb{C}^N \times R^{af}(1, d). 
\end{align*}
Let $\widetilde{W}$ be a super-potential on $Q^{af, N}$ satisfying \[\widetilde{W}|_{Q}=X[Y, Z],\] where 
$Q \subset Q^{af, N}$ is the full subquiver consisting 
of the vertex $\{1\}$, i.e. it is the triple loop quiver. 
We define 
\begin{align*}
    \mathcal{DT}^{a, N}_{\widetilde{W}}(d) &:= \mathrm{MF}\left(R^{af, N}(1, d)^{\chi_0 \text{-ss}}/GL(d), \Tr \widetilde{W}  \right), \\
    \mathcal{PT}^{a, N}_{\widetilde{W}}(d) &:= \mathrm{MF}\left(R^{af, N}(1, d)^{\chi_0^{-1} \text{-ss}}/GL(d), \Tr \widetilde{W}  \right).
\end{align*}
By applying $D^b(\mathbb{C}^N)\boxtimes$
to the semiorthogonal decomposition in Theorem~\ref{thm2bis}, and 
the super-potential $\Tr \widetilde{W}$, see~\cite[Proposition~2.1]{P0}, we have the following 
corollary: 
\begin{cor}\label{cor:variant}
Let $\mu \in \mathbb{R}$ such that 
$2\mu l \notin \mathbb{Z}$
for $1\leq l \leq d$. 
 There is a semiorthogonal decomposition 
    \begin{equation*}
  \mathcal{DT}^{a, N}_{\widetilde{W}}(d)=\Big\langle  \left( \boxtimes_{i=1}^k \mathbb{S}(d_i)_{w_i}\right)\boxtimes 
  \mathcal{PT}^{a, N}_{\widetilde{W}}(d')\Big\rangle.
   \end{equation*}
    The right hand side is after all $d'\leq d$, partitions $(d_i)_{i=1}^k$ of $d-d'$, and integers $(w_i)_{i=1}^k$ such that for $v_i:=w_i+d_i\left(d'+\sum_{j>i} d_j-\sum_{j<i}d_j\right)$, we have
    \begin{equation*}
    -1-\mu-\frac{a}{2}< \frac{v_1}{d_1}<\cdots<\frac{v_k}{d_k}<-\mu-\frac{a}{2}.
\end{equation*}
    The order of the semiorthogonal summands is the same as in Theorem~\ref{thm2bis}. 
\end{cor}

The analogous conclusion holds if we replace $\mathbb{C}^N$ by an open subset in all the constructions and statements above.

\section{The categorical DT/PT correspondence for \texorpdfstring{$\mathbb{C}^3$}{C3}}\label{sec:examplelocal}
In this section, we prove Theorems~\ref{thm:globalcrit}
and \ref{thm:exam} and Corollary~\ref{cor14}. 
We give explicit descriptions of DT/PT moduli 
spaces on $\mathbb{C}^3$
with reduced supports via extended ADHM quivers. 
The ADHM quiver 
is a quiver with a relation (depicted in Figure~3)
which 
was used to construct
framed moduli spaces of instantons on $\mathbb{P}^2$, 
see~\cite{NLecture}. 
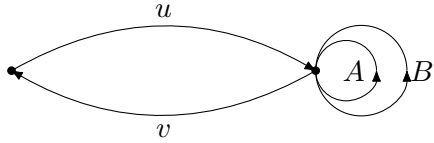
\begin{figure}[H]
\begin{align*}
	&\begin{tikzpicture}			
			\draw[->, >={Latex[round]}] 	
			(-4, 0) to [bend left=30] (0, 0);
					\draw[->, >={Latex[round]}] 	
				(0, 0) to [bend left=30] (-4, 0);
				\draw[->, >={Latex[round]}] (0, 0) arc (-180:0:0.4) ;
		\draw (0.8, 0) arc (0:180:0.4);
		\draw[->, >={Latex[round]}] (0, 0) arc (-180:0:0.6) ;
		\draw (1.2, 0) arc (0:180:0.6);
		\draw[fill=black] (0, 0) circle (0.05);
		\draw[fill=black] (-4, 0) circle (0.05);
			\draw (-2,0.8) node[circle] {$u$};
				\draw (-2,-0.8) node[circle] {$v$};
				\draw (0.5,0) node[circle] {$A$};
					\draw (1.4,0) node[circle] {$B$};
	\end{tikzpicture} \\
	&[A, B]+u \circ v=0. 
	\end{align*}
	\caption{ADHM quiver}
	\end{figure}

We use an interpretation of representations of 
the ADHM quiver in terms of perverse coherent sheaves, 
and then relate them to moduli spaces of 
pairs of one-dimensional sheaves with sections. 

\subsection{Perverse coherent sheaves via the ADHM construction}
Below we regard $\mathbb{C}^2$ as an open subset of 
$\mathbb{P}^2$ consisting of $[X \colon Y \colon Z]$ such that $Z \neq 0$, 
and set $l_{\infty}=(Z=0) \subset \mathbb{P}^2$. 

Let $\mathcal{T} \subset \Coh(\mathbb{P}^2)$ be the subcategory of zero-dimensional sheaves and let
$\mathcal{F} \subset \Coh(\mathbb{P}^2)$ be the subcategory of sheaves $F$ such that $\Hom(\mathcal{T}, F)=0$. 
The pair $(\mathcal{T}, \mathcal{F})$ is a torsion pair on 
$\Coh(\mathbb{P}^2)$, and we define $\mathcal{A}$ to be the associated
tilting abelian category
\begin{align*}
	\mathcal{A}=\langle \mathcal{F}, \mathcal{T}[-1] \rangle \subset D^b(\mathbb{P}^2). 
	\end{align*}
An object in $\mathcal{A}$ is an example of a perverse coherent sheaf on $\mathbb{P}^2$
introduced in~\cite{Kashi, ArBez}. 
We denote by $\mathfrak{U}_{\mathbb{C}^2}(d)$ the derived moduli 
stack of objects $I \in \mathcal{A}$ with an isomorphism 
$I|_{l_{\infty}}\cong \mathcal{O}_{l_{\infty}}$ such that $\ch_2(I)=-d$. 
Note that $I$ fits into an exact sequence in $\mathcal{A}$:
\begin{align}\label{exact:I}
0\to I_Z \to I \to Q[-1] \to 0,
	\end{align}
where $I_Z=\mathcal{H}^0(I)$ is an ideal sheaf of a zero-dimensional 
subscheme $Z \subset \mathbb{C}^2$ and 
$Q=\mathcal{H}^1(I)$ is zero-dimensional on $\mathbb{C}^2$
such that $\lvert Q \rvert+\lvert Z \rvert=d$. 
The stack $\mathfrak{U}_{\mathbb{C}^2}(d)$ is an Artin stack of finite type 
whose classical truncation $\mathfrak{U}^{\mathrm{cl}}_{\mathbb{C}^2}(d)$
admits a good moduli space
\begin{align*}
	\mathfrak{U}^{\mathrm{cl}}_{\mathbb{C}^2}(d) \to \mathrm{Sym}^d(\mathbb{C}^2),\, I\mapsto \mathrm{Supp}(Q)+\mathrm{Supp}(Z).
	\end{align*} 

There is an explicit description of $\mathfrak{U}^{\mathrm{cl}}_{\mathbb{C}^2}(d)$ 
via the ADHM quiver. 
Let $V$ be a $d$-dimensional vector space and let $\mathfrak{g}=\Hom(V, V)$. 
We consider the following $GL(V)$-equivariant morphism:
\begin{align*}
	\nu \colon V \oplus V^{\vee} \oplus \mathfrak{g}^{\oplus 2} &\to \mathfrak{g},\\
	(u, v, A, B) &\mapsto [A, B]+u \circ v.
	\end{align*}
Let $\nu^{-1}(0)$ be the derived zero locus of $\nu$. 
The derived moduli stack of the ADHM quiver with relation 
is given by 
\begin{align*}
	\mathfrak{N}_{\mathbb{C}^2}(d) :=\nu^{-1}(0)/GL(V). 
\end{align*}
Then there is an equivalence of derived stacks, 
see~\cite[Theorem~5.7]{BFG} 
	\begin{align}\label{Ups}
		\Upsilon \colon \mathfrak{N}_{\mathbb{C}^2}(d) \stackrel{\sim}{\to}
	\mathfrak{U}_{\mathbb{C}^2}(d).
	\end{align}
For a point $(u, v, A, B) \in \nu^{-1}(0)$, 
the corresponding object in $\mathfrak{U}_{\mathbb{C}^2}(d)$ is given by 
the complex on $\mathbb{P}^2$:
\begin{align}\label{perv:V}
	0 \to V\otimes \mathcal{O}_{\mathbb{P}^2}(-1)
	\stackrel{\phi}{\to} (V^{\oplus 2} \oplus \mathbb{C})\otimes \mathcal{O}_{\mathbb{P}^2}
	\stackrel{\psi}{\to} V \otimes \mathcal{O}_{\mathbb{P}^2}(1) \to 0.
	\end{align}
Here, the maps $\phi$ and $\psi$ are given by 
\begin{align}\label{alphabeta}
	\phi=\begin{pmatrix}
		ZA-X\id \\
		ZB-Y\id \\
		Zv
		\end{pmatrix}, \
	\psi=\begin{pmatrix}
		-ZB+Y\id, ZA-X\id, Zu
		\end{pmatrix}. 
	\end{align}
It is proved in~\cite[Theorem~5.7]{BFG} that the 
above correspondence gives an isomorphism (\ref{Ups}) on 
classical truncations.  
The above construction and its inverse in~\cite[Theorem~5.7]{BFG} can be generalized to families over dg-rings, 
so giving the claimed equivalence (\ref{Ups}). 

\subsection{The extended ADHM stack}
For each $m\in \mathbb{N}$, let $I_m$ be the set
\begin{align*}
	I_m:= \{(i, j) \in \mathbb{Z}_{\geq 0}^{\oplus 2} : 1\leq i+ j \leq m\}. 
	\end{align*}
Note that we have $\lvert I_m \rvert =m^2/2+3m/2=\dim \lvert \mathcal{O}_{\mathbb{P}^2}(m) \rvert$. 
For each $\alpha=(\alpha_{ij}) \in \mathbb{C}^{I_m}$, consider the polynomial
\begin{align}\label{def:falpha}
	f_{\alpha}=1+\sum_{(i, j) \in I_m}\alpha_{ij}x^i y^j \in \mathbb{C}[x, y]
	\end{align}
and let $C_{\alpha} \subset \mathbb{C}^2$ be the plane curve defined by $f_{\alpha}=0$. 
There is an open immersion 
\begin{align}\label{open:I}
	\mathbb{C}^{I_m} \hookrightarrow \lvert \mathcal{O}_{\mathbb{P}^2}(m) \rvert, \ 
	\alpha \mapsto \overline{f}_{\alpha}:=Z^m f_{\alpha}(X/Z, Y/Z). 
\end{align}
whose image consists of polynomials
$\sum c_{ijk}X^i Y^j Z^k$ such that $c_{00m} \neq 0$. 

Let $V$ be a $d$-dimensional vector space
and $\mathfrak{g}=\Hom(V, V)$. 
We consider the map 
\begin{align*}
	\mu \colon 
	V \oplus V^{\vee} \oplus \mathfrak{g}^{\oplus 2} \oplus 
	\mathbb{C}^{I_m} &\to \mathfrak{g} \oplus V^{\vee}\\
	(u, v, A, B, \alpha) &\mapsto 
	([A, B]+u \circ v, v \circ f_{\alpha}(A, B)).
	\end{align*}
The above map $\mu$ is $GL(V)$-equivariant, where $GL(V)$ acts on $\mathfrak{g}$ by 
conjugation and on $\mathbb{C}^{I_m}$ trivially. 
Let $\mu^{-1}(0)$ be the derived zero locus of $\mu$. 
We define the derived stack 
\begin{align*}
	\mathfrak{M}_{\mathbb{C}^2}(m, d) :=\mu^{-1}(0)/GL(V). 
	\end{align*}
By forgetting $\alpha$, there is a morphism of 
derived stacks 
\begin{align}\label{g}
	g \colon \mathfrak{M}_{\mathbb{C}^2}(m, d) \to \mathfrak{N}_{\mathbb{C}^2}(d). 
	\end{align}
In what follows, we explain that $\mathfrak{M}_{\mathbb{C}^2}(m, d)$
can be interpreted as a derived moduli stack of pairs of a one-dimensional sheaf on $\mathbb{C}^2$
with a section. 

\subsection{The derived moduli stack of pairs on \texorpdfstring{$\mathbb{C}^2$}{C2}}
We 
consider pairs 
\begin{align}\label{pairs}
	(F, s), \ F \in \Coh(\mathbb{P}^2), s \colon \mathcal{O}_{\mathbb{P}^2} \to F,
	\end{align}
where $F$ is one-dimensional 
and the cokernel of $s$ is at most zero-dimensional. Let  $l\subset \mathbb{P}^2$ be a line.
We denote by 
\begin{align*}
	\mathfrak{T}_{\mathbb{P}^2}(m, d) \hookleftarrow \mathfrak{T}^{\mathrm{cl}}_{\mathbb{P}^2}(m, d)
	\end{align*}
the derived moduli stack 
of pairs (\ref{pairs}) and its classical truncation
satisfying 
\begin{align*}
	[F]=m[l], \chi(F)=3m/2-m^2/2+d.\end{align*}
The stack $\mathfrak{T}^{\mathrm{cl}}_d(\mathbb{P}^2, m)$
is an Artin stack of finite type
which admits  
a good moduli space, see~\cite[Subsection~4.2.4]{T}. 

Note that a pair $(\mathcal{O}_{\mathbb{P}^2} \to F)$
in $\mathfrak{T}_{\mathbb{P}^2}(m, d)$
admits a filtration
\begin{align*}
	(0 \to Q) \subset (\mathcal{O}_{\mathbb{P}^2} \twoheadrightarrow F')
	\subset (\mathcal{O}_{\mathbb{P}^2} \to F)
	\end{align*}
such that $Q$ and $Q':=F/F'$ are zero-dimensional sheaves with 
$\lvert Q \rvert+\lvert Q' \rvert=d$
and $F'/Q=\mathcal{O}_{C}$ for $C \in \lvert \mathcal{O}_{\mathbb{P}^2}(m) \rvert$. 
The good moduli space morphism is 
\begin{align*}
\mathfrak{T}^{\mathrm{cl}}_{\mathbb{P}^2}(m, d)
	&\to T_{\mathbb{P}^2}(m, d)=\lvert \mathcal{O}_{\mathbb{P}^2}(m) \rvert 
	\times \mathrm{Sym}^d(\mathbb{P}^2), \\
	(F, s)& \mapsto  (C, \mathrm{Supp}(Q)+\mathrm{Supp}(Q')). 
	\end{align*}
We define the derived open substack 
$\mathfrak{T}_{\mathbb{C}^2}(m, d) \subset \mathfrak{T}_{\mathbb{P}^2}(m, d)$
whose classical truncation fits into the Cartesian square 
\begin{align*}
	\xymatrix{
\mathfrak{T}_{\mathbb{C}^2}^{\mathrm{cl}}(m, d) \ar@<-0.3ex>@{^{(}->}[r]\ar[d] & \mathfrak{T}^{\mathrm{cl}}_{\mathbb{P}^2}(m, d) \ar[d] \\
\mathbb{C}^{I_m} \times \mathrm{Sym}^d(\mathbb{C}^2)  \ar@<-0.3ex>@{^{(}->}[r] & 
\lvert \mathcal{O}_{\mathbb{P}^2}(m) \rvert 
\times \mathrm{Sym}^d(\mathbb{P}^2). 
}
	\end{align*}
\begin{lemma}
	There is a natural morphism 
	\begin{align}\label{mor:h}
		h \colon \mathfrak{T}_{\mathbb{C}^2}(m, d) \to 
		\mathfrak{U}_{\mathbb{C}^2}(d)
				\end{align}
			sending a pair $(F, s)$ to a two term complex
			$(\mathcal{O}_{\mathbb{P}^2}(m) \stackrel{s}{\to}F(m))$. 
	\end{lemma}
\begin{proof}
	By the definition of $\mathfrak{T}_{\mathbb{C}^2}(m, d)$, 
	for a point $(F, s)$ in $\mathfrak{T}_{\mathbb{C}^2}(m, d)$,
	the 
	two term complex
	\[I=(\mathcal{O}_{\mathbb{P}^2}(m) \stackrel{s}{\to}F(m))\]
	is an object in $\mathcal{A}$ such that 
	$\mathcal{H}^0(I)=I_Z$ for a zero-dimensional subscheme $Z \subset \mathbb{C}^2$ and $\mathcal{H}^1(I)$ is zero-dimensional with 
	support contained in $\mathbb{C}^2$. 
	In particular, there is an open neighborhood $l_{\infty} \subset U \subset
	\mathbb{P}^2$ such that 
	$I|_{U}=(\mathcal{O}_U(m) \stackrel{s|_{U}}{\to} \mathcal{O}_C(m)|_U)$
	for $C \in \lvert \mathcal{O}_{\mathbb{P}^2}(m) \rvert$. Then 
	$I|_{U}$ is naturally isomorphic to $\mathcal{O}_U$, giving a 
	trivialization $I|_{l_{\infty}} \cong \mathcal{O}_{l_{\infty}}$. 
	\end{proof}

\begin{lemma}\label{lem:HomI}
	For an object $I$ in $\mathfrak{U}_{\mathbb{C}^2}(d)$, 	we have 
	\begin{align*}
		\mathrm{Hom}_{\mathbb{P}^2}(I, \mathcal{O}_{\mathbb{P}^2}(m))
		=\Ker\left(H^0(\mathcal{O}_{\mathbb{P}^2}(m)) \stackrel{\eta}{\to}
		\mathrm{Ext}^2_{\mathbb{P}^2}(\mathcal{H}^1(I), \mathcal{O}_{\mathbb{P}^2}(m))   \right).
		\end{align*}
	In the above, $\eta$ sends $a \in H^0(\mathcal{O}_{\mathbb{P}^2}(m))$ 
	to $\mathcal{H}^1(I) \to I_Z[2] \to \mathcal{O}_{\mathbb{P}^2}(m)[2]$
	where the first arrow is the extension class of (\ref{exact:I}) and 
	the second arrow is given by multiplication with $a$. 
	\end{lemma}
\begin{proof}
	The lemma is straightforward by applying $\mathrm{RHom}_{\mathbb{P}^2}(-, \mathcal{O}_{\mathbb{P}^2}(m))$ to (\ref{exact:I})
	and observing that 
	$\mathrm{Hom}(I_Z, \mathcal{O}_{\mathbb{P}^2}(m))=H^0(\mathcal{O}_{\mathbb{P}^2}(m))$. 
	\end{proof}
By Lemma~\ref{lem:HomI},
we have an injection 
\begin{align}\label{inj:I}
\Hom(I, \mathcal{O}_{\mathbb{P}^2}(m)) \hookrightarrow 
H^0(\mathcal{O}_{\mathbb{P}^2}(m))
	\end{align}
which sends $t \colon I \to \mathcal{O}_{\mathbb{P}^2}(m)$ to $\det(t)
 \colon \det I=\mathcal{O}_{\mathbb{P}^2} \to \mathcal{O}_{\mathbb{P}^2}(m)$. 
We set
\begin{align*}
\mathbb{P}(\Hom(I, \mathcal{O}_{\mathbb{P}^2}(m)))^{\circ}:=
\mathbb{P}(\Hom(I, \mathcal{O}_{\mathbb{P}^2}(m))) \cap \mathbb{C}^{I_m}. 
\end{align*}
\begin{lemma}\label{lem:class}
	For an object $I$ in 
	$\mathfrak{U}_{\mathbb{C}^2}(d)$, 
	we have 
	$h^{-1}(I)^{\mathrm{cl}} \cong \mathbb{P}(\Hom(I, \mathcal{O}_{\mathbb{P}^2}(m)))^{\circ}$. 
	\end{lemma}
 \begin{proof}
Given $I$, let us take a non-zero morphism $t \colon I \to \mathcal{O}_{\mathbb{P}^2}(m)$
corresponding to a point in 
$\mathbb{P}(\Hom(I, \mathcal{O}_{\mathbb{P}^2}(m)))^{\circ}$. 
Let $C(t)$ be the cone of $t$ and consider the distinguished triangle
\begin{align*}
	I \to \mathcal{O}_{\mathbb{P}^2}(m) \to C(t)\to I[1]. 
	\end{align*}
By taking the associated long exact sequence of cohomologies, 
we see that $C(t) \in \mathrm{Coh}(\mathbb{P}^2)$
and it fits into the exact 
sequence
\begin{align*}
	0 \to \mathcal{O}_{\mathbb{P}^2}(m)/\mathcal{H}^0(I) \to 
	C(t) \to \mathcal{H}^1(I) \to 0. 
	\end{align*}
It follows that $C(t)$ is one-dimensional with support a curve 
represented by an element in $\mathbb{C}^{I_m}$. Moreover, $I$ is isomorphic 
to the two-term complex
\[I\cong (\mathcal{O}_{\mathbb{P}^2}(m) \xrightarrow{s} C(t))\] 
such that $\mathrm{Cok}(s)=\mathcal{H}^1(I)$ is zero-dimensional. 
Suppose that for another morphism $t':I\to \mathcal{O}_{\mathbb{P}^2}(m)$, there is a commutative diagram 
\begin{align*}
	\xymatrix{
I \ar[r]^-{t} & \mathcal{O}_{\mathbb{P}^2}(m) \ar[r] \ar@{=}[d] & C(t) \ar[d]^-{\cong} \\	
I \ar[r]^-{t'} & \mathcal{O}_{\mathbb{P}^2}(m) \ar[r] & C(t'). 
}
	\end{align*}
Then there is an isomorphism $\phi \colon I \stackrel{\cong}{\to} I$
which makes the above diagram commutative. 
Then $ \det t \det \phi =\det t'$.
As $0\neq \det \phi \in H^0(\mathcal{O}_{\mathbb{P}^2})=\mathbb{C}$, 
and (\ref{inj:I}) is injective, 
we see that $t$ and $t'$ differ by a non-zero scalar multiplication. 
It follows that 
\begin{align}\label{inclu:h}
\mathbb{P}(\Hom(I, \mathcal{O}_{\mathbb{P}^2}(m)))^{\circ} \subset h^{-1}(I)^{\mathrm{cl}}. 
	\end{align}

Conversely, let $(\mathcal{O}_{\mathbb{P}^2}\xrightarrow{s} F)$ be a point 
in $\mathfrak{T}_{\mathbb{C}^2}(m, d)$ such that the two-term 
complex $(\mathcal{O}_{\mathbb{P}^2}(m) \stackrel{s}{\to} F(m))$ is isomorphic to 
$I$ in $D^b(\mathbb{P}^2)$. 
Then there is a non-zero morphism $I \to \mathcal{O}_{\mathbb{P}^2}(m)$ 
corresponding to a point in $\mathbb{P}(\Hom(I, \mathcal{O}_{\mathbb{P}^2}(m)))^{\circ}$. 
Therefore (\ref{inclu:h}) is an isomorphism. 
 	\end{proof}
 \subsection{Explicit description of the moduli stack of pairs on \texorpdfstring{$\mathbb{C}^2$}{C2}}
 For $(i, j) \in I_m$ and $(A, B) \in \mathfrak{g}^{\oplus 2}$, we have 
 \begin{align*}
 	A^i B^j-x^i y^j \id=& 
 	A^i(B^{j-1}+yB^{j-2}+\cdots+y^{j-1}\id)(B-y\id)\\
 	&+y^j(A^{i-1}+xA^{i-2}+\cdots+x^{i-1}\id)(A-x\id). 
 	\end{align*}
 Let $\alpha \in \mathbb{C}^{I_m}$.
 By the above equality, we 
 can write 
 \begin{align*}
 	f_{\alpha}(A, B)-f_{\alpha} \id=
 	g_{\alpha}(A-x \id)+h_{\alpha}(B-y\id)
 	\end{align*}
 for $g_{\alpha}, h_{\alpha} \in \mathbb{C}[x, y] \otimes \mathfrak{g}$ 
 with degree less than $m$. 
 Then 
 for each element 
 \begin{align*}
     (u, v, A, B, \alpha) \in V \oplus V^{\vee} \oplus \mathfrak{g}^{\oplus 2} \oplus \mathbb{C}^{I_m}
     \end{align*}
 such that $[A, B]+u \circ v=0$
 and $v \circ f_{\alpha}(A, B)=0$, 
 we have the commutative diagram 
 \begin{align}\label{dia:V}
 	\xymatrix{
 V \otimes \mathcal{O}_{\mathbb{P}^2}(-1) \ar[r]^-{\phi} \ar[d] & 
 (V^{\oplus 2} \oplus \mathbb{C}) \otimes \mathcal{O}_{\mathbb{P}^2} \ar[r]^-{\psi} \ar[d]_-{\gamma} & 
 V \otimes \mathcal{O}_{\mathbb{P}^2}(1) \ar[d] \\
 0 \ar[r] & \mathcal{O}_{\mathbb{P}^2}(m) \ar[r] & 0. 	
 }
 	\end{align}
 In the above diagram, $\phi, \psi$ are given as (\ref{alphabeta}), and $\gamma$
 is given by 
 \begin{align*}
 \gamma=(v \circ \overline{g}_{\alpha}, v \circ \overline{h}_{\alpha}, \overline{f}_{\alpha}). 	
 	\end{align*}
 Here $\overline{g}_{\alpha}=Z^m g_{\alpha}(X/Z, Y/Z) \in \mathbb{C}[X, Y, Z] \otimes 
 \mathfrak{g}$, etc. 
 We note that the equation $v \circ f_{\alpha}(A, B)=0$
 is used for the commutativity of the left 
 square in (\ref{dia:V}). 
 
 The diagram (\ref{dia:V})
 determines a morphism 
 $I \to \mathcal{O}_{\mathbb{P}^2}(m)$, where $I$ is a perverse
 coherent sheaf represented by (\ref{perv:V}). 
 From the proof of Lemma~\ref{lem:class}, 
 the totalization of the diagram (\ref{dia:V})
 is quasi-isomorphic to a one-dimensional sheaf $F$ on $\mathbb{P}^2$ with 
 curve class $(\overline{f}_{\alpha})=0$ and
 with a morphism 
 $\mathcal{O}_{\mathbb{P}^2}(m) \to F$. 
 Therefore, by applying $\otimes \mathcal{O}_{\mathbb{P}^2}(-m)$, we obtain 
 the morphism 
  \begin{align*}
 	\Theta \colon \mathfrak{M}_{\mathbb{C}^2}(m, d) \to \mathfrak{T}_{\mathbb{C}^2}(m, d)
 	\end{align*}
 such that the following diagram commutes: 
 \begin{align}\label{com:Ups}
 	\xymatrix{ \mathfrak{M}_{\mathbb{C}^2}(m, d) \ar[r]^-{\Theta} \ar[d]_-{g} &  \mathfrak{T}_{\mathbb{C}^2}(m, d)
 		\ar[d]^-{h} \\
 		\mathfrak{N}_{\mathbb{C}^2}(d) \ar[r]_-{\Upsilon}^-{\sim} & \mathfrak{U}_{\mathbb{C}^2}(d).  	
 }
 	\end{align}
\begin{prop}
	The morphism $\Theta$ is an equivalence of derived stacks. 
	\end{prop}	   
\begin{proof}
	We first prove that $\Theta$ is an isomorphism on classical truncations. 
	Let $(u, v, A, B)$ be a point in $\mathfrak{U}_{\mathbb{C}^2}(d)$
	and $I$ be the associated perverse coherent sheaf 
	via the construction (\ref{perv:V}). 
	From the commutative diagram (\ref{com:Ups}), 
	it is enough to show that $\Theta$ induces the isomorphism 
	\begin{align}\label{theta:l}
		\Theta^{\rm{cl}} \colon g^{-1}(u, v, A, B)^{\rm{cl}} \stackrel{\cong}{\to}
		h^{-1}(I)^{\rm{cl}}. 
		\end{align}
	Here, the right hand side is given by 
	$\mathbb{P}(\Hom(I, \mathcal{O}_{\mathbb{P}^2}(m)))^{\circ}$ by Lemma~\ref{lem:class}
	and the left hand side is given by 
	the zero-locus of the map 
	\begin{align}\label{map:Im}
		\mathbb{C}^{I_m} \to V^{\vee}, \ 
		\alpha \mapsto v \circ f_{\alpha}(A, B). 
		\end{align} 
	Let $\Sigma \subset V$ be the intersection of all 
	subspaces $S \subset V$ satisfying 
	$A(S) \subset S$, $B(S) \subset S$ and $u \in S$. 
	We set $N=V/\Sigma$. Then we have 
	the commutative diagram 
	\begin{align}\label{dia:sigma}
		\xymatrix{
	\Sigma \otimes \mathcal{O}_{\mathbb{P}^2}(-1) \ar[r]^-{\phi|_{\Sigma}} \ar[d] & 
	(\Sigma^{\oplus 2} \oplus \mathbb{C}) \otimes \mathcal{O}_{\mathbb{P}^2} \ar[r]^-{\psi|_{\Sigma}} \ar[d]^-{\delta} 
	& \Sigma \otimes \mathcal{O}_{\mathbb{P}^2}(1) \ar[d] \\
	V \otimes \mathcal{O}_{\mathbb{P}^2}(-1) \ar[r]^-{\phi} \ar[d] & 
	(V^{\oplus 2} \oplus \mathbb{C}) \otimes \mathcal{O}_{\mathbb{P}^2} \ar[r]^-{\psi} \ar[d] 
	& V \otimes \mathcal{O}_{\mathbb{P}^2}(1)	\ar[d] \\
	N \otimes \mathcal{O}_{\mathbb{P}^2}(-1) \ar[r] & 
	N^{\oplus 2} \otimes \mathcal{O}_{\mathbb{P}^2} \ar[r] 
	& N \otimes \mathcal{O}_{\mathbb{P}^2}(1).
	}
		\end{align}
		Recall that the middle horizontal complex is isomorphic to $I$.
	By~\cite[Theorem~4.10]{Jar2}, the top and bottom horizontal complexes are isomorphic to $I_Z$ and $Q[-1]$, respectively, and the vertical arrows between these complexes are identified with the
	exact sequence (\ref{exact:I}) in $\mathcal{A}$. Moreover, we have $v|_{\Sigma}=0$
	by~\cite[Proposition~2.8]{NLecture}. 
	Take a non-zero element 
	\begin{align*}
		c=\sum_{i+j+k=m}c_{ijk} X^i Y^j Z^k \in H^0(\mathcal{O}_{\mathbb{P}^2}(m)). 
		\end{align*}
	Under the identification $\Hom(I_Z, \mathcal{O}_{\mathbb{P}^2}(m))=
	H^0(\mathcal{O}_{\mathbb{P}^2}(m))$, the morphism 
	$I_Z \to \mathcal{O}_{\mathbb{P}^2}(m)$ corresponding to $c$ is represented by 
	the commutative diagram 
	\begin{align}\label{diag66}
		\xymatrix{\Sigma \otimes \mathcal{O}_{\mathbb{P}^2}(-1) \ar[r]^-{\phi|_{\Sigma}} \ar[d] & 
			(\Sigma^{\oplus 2} \oplus \mathbb{C}) \otimes \mathcal{O}_{\mathbb{P}^2} \ar[r]^-{\psi|_{\Sigma}} \ar[d]_-{(0, 0, c)}
			& \Sigma \otimes \mathcal{O}_{\mathbb{P}^2}(1) \ar[d] \\
			0 \ar[r] & \mathcal{O}_{\mathbb{P}^2}(m) \ar[r] & 0. 	
	}
		\end{align}
	We compute $\eta(c)$ in Lemma~\ref{lem:HomI} 
	using the diagram (\ref{dia:sigma}). 
	Since $\mathcal{H}^1(I)$ is supported on $\mathbb{C}^2$, 
	via the isomorphism $Z^m \colon \mathcal{O}_{\mathbb{C}^2} \stackrel{\cong}{\to}
	\mathcal{O}_{\mathbb{P}^2}(m)|_{\mathbb{C}^2}$
	we have 
	\begin{align}\label{isom:Z}
		\Ext^2_{\mathbb{P}^2}(\mathcal{H}^1(I), \mathcal{O}_{\mathbb{P}^2}(m)) \cong 
		H^0(\mathcal{E}xt^2_{\mathbb{C}^2}(\mathcal{H}^1(I), \mathcal{O}_{\mathbb{C}^2})) \cong N^{\vee}
		\subset V^{\vee}. 		
		\end{align}
	In order to compute $\eta(c)$ as an element of $V^{\vee}$, 
	we extend the middle vertical arrow in (\ref{diag66})
	to $\varepsilon$
	as follows
	\begin{align*}
	    \xymatrix{
	    (\Sigma^{\oplus 2} \oplus \mathbb{C}) \otimes \mathcal{O}_{\mathbb{P}^2} \ar[d]_-{(0, 0, c)} 
	    \ar[r]^-{\delta} & 
	     (V^{\oplus 2} \oplus \mathbb{C}) \otimes \mathcal{O}_{\mathbb{P}^2} \ar[ld]^-{\varepsilon=(0, 0, c)} \\
	     \mathcal{O}_{\mathbb{P}^2}(m). 
	    }
	\end{align*}
		We then compose it with the morphism $\phi$ in (\ref{dia:sigma})
	to obtain the morphism 
	\begin{align*}
	    V \otimes \mathcal{O}_{\mathbb{P}^2}(-1) \stackrel{\phi}{\to}  (V^{\oplus 2} \oplus \mathbb{C}) \otimes \mathcal{O}_{\mathbb{P}^2} \stackrel{\varepsilon}{\to} \mathcal{O}_{\mathbb{P}^2}(m). 
	\end{align*}
	The morphism above factors through $N \otimes \mathcal{O}_{\mathbb{P}^2}(-1)$, giving 
	a morphism 
	\begin{align*}
	v \circ Zc \colon	N \otimes \mathcal{O}_{\mathbb{P}^2}(-1)
	\to \mathcal{O}_{\mathbb{P}^2}(m).
	\end{align*}
	The above morphism represents 
	$\eta(c)$. Under the isomorphism (\ref{isom:Z}), 
	we have 
	\begin{align*}
		\eta(c)=v \circ \sum_{i+j+k=m}c_{ijk}A^i B^j. 
		\end{align*}
	Note that we have $\alpha_{ij}=c_{ijk}/c_{00m}$ under the 
	embedding (\ref{open:I}), where $k=m-i-j$. 
	Therefore comparing with (\ref{map:Im}), 
	we see the isomorphism (\ref{theta:l}). 
	
	Finally, we show that $\Theta$ induces an equivalence of derived stacks. 
	As $\Theta$ is an isomorphism on classical truncations and the 
	bottom arrow in (\ref{com:Ups}) is an equivalence, it is enough 
	to show that $\Theta$ induces a quasi-isomorphism 
	on relative tangent complexes 
	$\mathbb{T}_g \stackrel{\sim}{\to}\Theta^{\ast}\mathbb{T}_h$. 
	Let us take a point $x=(u, v, A, B, \alpha)$ in $\mathfrak{M}_{\mathbb{C}^2}(m, d)$
	such that $[A, B]+u \circ v=0$ and $v \circ f_{\alpha}(A, B)=0$. 
	Then we have 
	\begin{align*}
		\mathbb{T}_g|_{x}=(\mathbb{C}^{I_m} \stackrel{\iota}{\to} V^{\vee}), \ 
		\iota \colon 
		(e_{ij}) \mapsto \sum_{(i, j) \in I_m} e_{ij} v\circ A^i B^j. 
		\end{align*}
		Let $y=(\mathcal{O}_{\mathbb{P}^2} \to F)$ be the 
	point in $\mathfrak{T}_{\mathbb{C}^2}(m, d)$ corresponding to 
	$\Theta(x)$. The tangent complex of $\mathfrak{T}_{\mathbb{C}^2}(m, d)$
	at $y$ 
	is 
	\begin{align*}
		\mathrm{RHom}(\mathcal{O}_{\mathbb{P}^2} \to F, F)
		=\mathrm{RHom}(I, F(m)),
		\end{align*}
	where $I=(\mathcal{O}_{\mathbb{P}^2}(m) \to F(m))$. 
The tangent complex $\mathbb{T}_h|_{y}$ is 
\begin{align*}
	\mathrm{Cone}\left(\mathrm{RHom}(I, F(m)) \to \mathrm{RHom}(I, I)_0[1]\right)
	\cong \mathrm{Cone}\left(\mathbb{C} \to \mathrm{RHom}(I, \mathcal{O}_{\mathbb{P}^2}(m))\right). 
	\end{align*}
Here, $\mathrm{RHom}(I, I)_0$ is the cone of the natural morphism
$\mathbb{C} \to \mathrm{RHom}(I, I)$, and 
the map $\mathbb{C} \to \mathrm{RHom}(I, \mathcal{O}_{\mathbb{P}^2}(m))$ 
is the composition of the above natural morphism 
with $I \to \mathcal{O}_{\mathbb{P}^2}(m)$. 
From the calculation of $\eta$, we have 
\begin{align*}
	\mathbb{T}_{h}|_{y}=\left(H^0(\mathcal{O}_{\mathbb{P}^2}(m))/\mathbb{C} 
	\overline{f}_{\alpha}
	\stackrel{\iota'}{\to} V^{\vee}   \right), \ 
\iota' \colon	(c_{ijk}) \mapsto 
v \circ \sum_{i+j+k=m}c_{ijk} A^i B^j. 
	\end{align*}
The required quasi-isomorphism is 
\begin{align*}
	\xymatrix{
\mathbb{C}^{I_m} \ar[r]^-{\iota} \ar[d]_-{\cong} & V^{\vee} \ar@{=}[d]	\\
H^0(\mathcal{O}_{\mathbb{P}^2}(m))/\mathbb{C}\overline{f}_{\alpha} \ar[r]^-{\iota'} & V^{\vee}. 
}
	\end{align*}
Here the left vertical arrow sends $(e_{ij})$ to 
$(c_{ijk})$ for $c_{ijk}=e_{ij}$ when $i+j>0$ and $c_{00m}=0$. 
	\end{proof}

	\subsection{The categorical DT/PT correspondence for \texorpdfstring{$\mathbb{C}^3$}{C3}}\label{sub:45}
		We define 
	Zariski open subsets 
	\begin{align}\label{ired2}
		\lvert \mathcal{O}_{\mathbb{P}^2}(m) \rvert^{\rm{red}} \subset
		\lvert \mathcal{O}_{\mathbb{P}^2}(m) \rvert, \ 
		(\mathbb{C}^{I_m})^{\rm{red}} :=	\lvert \mathcal{O}_{\mathbb{P}^2}(m) \rvert^{\rm{red}}
		\cap \mathbb{C}^{I_m}, 
	\end{align}
	where $\lvert \mathcal{O}_{\mathbb{P}^2}(m) \rvert^{\rm{red}}$ is the locus of reduced 
	plane
	curves in $\mathbb{P}^2$. 
	
Let 
$X=\mathrm{Tot}_{\mathbb{P}^2}(\omega_{\mathbb{P}^2})$. Let 
\begin{align*}
\mathcal{T}_X^{\rm{red}}(m, d)
\end{align*} be the classical moduli stack of pairs $(F, s)$, where $F$
is a one-dimensional sheaf on $X$ with support a reduced plane curve
of degree $m$ 
in $\mathbb{P}^2$, and $s \colon \mathcal{O}_X \to F$ is surjective in dimension one. 
The open immersion $\mathbb{C}^2 \subset \mathbb{P}^2$ determines 
the open immersion $\mathbb{C}^3 \subset X$. 
We have the open substack 
\begin{align*}
	\mathcal{T}_{\mathbb{C}^3}^{\rm{red}}(m, d)
	\subset \mathcal{T}_X^{\rm{red}}(m, d)
\end{align*}
to be consisting of $(F, s)$ such that 
$\mathrm{Cok}(s)$ and $F_{\rm{tor}}$ are supported on $\mathbb{C}^3$, 
and the one-dimensional support of $F$ is an element from $(\mathbb{C}^{I_m})^{\rm{red}}$. 
Here $F_{\rm{tor}} \subset F$ is the maximal zero-dimensional subsheaf. 

Let $\Tr W_{m, d}$ be the function 
\begin{align*}
	\Tr W_{m, d} \colon 
	\left(V^{\oplus 2} \oplus V^{\vee} \oplus \mathfrak{g}^{\oplus 3} \times (\mathbb{C}^{I_m})^{\rm{red}}\right)\big/GL(V)  \to \mathbb{C}
	\end{align*}
defined by the formula 
\begin{align}\label{W:formula}
(u_1, u_2, v, A, B, C, \alpha) \mapsto 
	v \circ f_{\alpha}(A, B)(u_2)+ \Tr C(u_1 \circ v+[A, B]). 
	\end{align}
	
	The next theorem describes $\mathcal{T}_{\mathbb{C}^3}^{\rm{red}}(m, d)$ as a critical locus. 
	Here it is essential to restrict to reduced curves, 
	see Remark~\ref{rmk:reduced}. 
\begin{thm}\label{lemma:globalcritical}
	We have
	\begin{align}\label{isom:critW}
		\mathcal{T}_{\mathbb{C}^3}^{\rm{red}}(m, d) \cong
		\mathrm{Crit}(\Tr W_{m, d}). 
		\end{align}
	\end{thm}
\begin{proof}
		We set 
	\begin{align*}
		\mathfrak{T}_{\mathbb{P}^2}^{\rm{red}}(m, d) &:=
		\mathfrak{T}_{\mathbb{P}^2}(m, d)\times_{\lvert \mathcal{O}_{\mathbb{P}^2}(m) \rvert}	\lvert \mathcal{O}_{\mathbb{P}^2}(m) \rvert^{\rm{red}}, \\
		\mathfrak{T}_{\mathbb{C}^2}^{\rm{red}}(m, d) &:=
		\mathfrak{T}_{\mathbb{C}^2}(m, d)\times_{\mathbb{C}^{I_m}} (\mathbb{C}^{I_m})^{\rm{red}}. 
	\end{align*}
By the proof of~\cite[Lemma~5.5.4]{T}, we have an isomorphism
\begin{align}\label{isom:reduced}
	\mathcal{T}_X^{\rm{red}}(m, d) \cong 
	\Omega_{\mathfrak{T}_{\mathbb{P}^2}^{\rm{red}}(m, d)}[-1]^{\rm{cl}}. 
	\end{align}
	The above isomorphism restricts to the isomorphism of open substacks 
	\begin{align*}
		\mathcal{T}_{\mathbb{C}^3}^{\rm{red}}(m, d) \cong 
	\Omega_{	\mathfrak{T}_{\mathbb{C}^2}^{\rm{red}}(m, d)}[-1]^{\rm{cl}}. 	
		\end{align*}
		The lemma follows from the explicit description of 
	the $(-1)$-shifted cotangent stack, see~\cite[Section~2.1.1]{T}. 
	\end{proof}
	\begin{remark}\label{rmk:reduced}
	We need to restrict to reduced curves in the above theorem 
	as we use the isomorphism (\ref{isom:reduced})
	from~\cite[Lemma~5.5.4]{T}.
	Otherwise, a pair $(F, s)$ from $X$ with zero-dimensional cokernel 
	may push-forward to a pair on $\mathbb{P}^2$ with one-dimensional cokernel, 
	so the isomorphism (\ref{isom:reduced}) may not hold. 
	\end{remark}

We have the open substacks, which are quasi-projective schemes:
\begin{align*}
	I_{\mathbb{C}^3}^{\rm{red}}(m, d) \subset 
	\mathcal{T}_{\mathbb{C}^3}^{\rm{red}}(m, d)
	\supset P_d^{\rm{red}}(\mathbb{C}^3, m).
	\end{align*}
	Here $I_d^{\rm{red}}(\mathbb{C}^3, m)$
	is the DT moduli space
	consisting of $(F, s)$ such that $s$ is 
	surjective, and
	$P_d^{\rm{red}}(\mathbb{C}^3, m)$ is the
	PT moduli space consisting of $(F, s)$
	such that $F$ is pure. 
Let $\chi_0 \colon GL(V) \to \mathbb{C}^{\ast}$ be the determinant character
$g \mapsto \det g$. 
Under the isomorphism (\ref{isom:critW}), 
the above open substacks correspond to GIT-semistable loci, see~\cite[Proposition~5.5.2]{T}:
\begin{align*}
	I_{\mathbb{C}^3}^{\rm{red}}(m, d)=
	\mathcal{T}_{\mathbb{C}^3}^{\rm{red}}(m, d)^{\chi_0\text{-ss}}, \
		P_{\mathbb{C}^3}^{\rm{red}}(m, d)=
	\mathcal{T}_{\mathbb{C}^3}^{\rm{red}}(m, d)^{\chi_0^{-1}\text{-ss}}. 
	\end{align*}
We define 
\begin{align}\label{regularfunction}
	\Tr W_{m, d}^{\pm} \colon 
	\left((V^{\oplus 2} \oplus V^{\vee} \oplus \mathfrak{g}^{\oplus 3})^{\chi_0^{\pm}\text{-ss}} \times (\mathbb{C}^{I_m})^{\rm{red}}\right)\big/GL(V) \to \mathbb{C}
\end{align}	
to be the restriction of $\Tr W_{m, d}$ to the GIT semistable loci. 
Then we have 
\begin{align*}
	I_{\mathbb{C}^3}^{\rm{red}}(m, d)=\mathrm{Crit}(\Tr W_{m, d}^+), \ 
		P_{\mathbb{C}^3}^{\rm{red}}(m, d)=\mathrm{Crit}(\Tr W_{m, d}^-). 
	\end{align*}
We define the following dg-categories
\begin{align*}
	\mathcal{DT}_{\mathbb{C}^3}^{\rm{red}}(m, d) &:=
	\mathrm{MF}\left((V^{\oplus 2} \oplus V^{\vee} \oplus \mathfrak{g}^{\oplus 3})^{\chi_0\text{-ss}} \times (\mathbb{C}^{I_m})^{\rm{red}}\big/GL(V), \Tr W_{m, d}^+ \right), \\
		\mathcal{PT}_{\mathbb{C}^3}^{\rm{red}}(m, d) &:=
	\mathrm{MF}\left((V^{\oplus 2} \oplus V^{\vee} \oplus \mathfrak{g}^{\oplus 3})^{\chi_0^{-1}\text{-ss}} \times (\mathbb{C}^{I_m})^{\rm{red}}\big/GL(V) , \Tr W_{m, d}^- \right). 
	\end{align*}
	The above dg-categories categorify DT invariants (resp.~PT invariants)
	for Hilbert schemes of 1-dimensional subschemes (resp.~PT stable pair 
	moduli spaces) with reduced 
	supports, 
	see~\cite[Theorem~1.1]{Eff}, \cite[Subsection 3.3]{T}. 
	\begin{remark}
	Following the definition of DT and PT categories in \cite{T}, one can also consider graded categories of matrix factorizations, with grading given by the 
	weight two $\mathbb{C}^{\ast}$-action on $v$ and $C$. The graded version also 
	recovers DT and PT invariants and 
	they have the same Grothendieck group as the ungraded one, 
	see~\cite[Proposition~3.3.6]{T}, \cite[Corollary~3.13]{T4}. 
	\end{remark}

\begin{proof}[Proof of Theorem \ref{thm:exam}]
The statement follows from Corollary~\ref{cor:variant} for $a=1$
and $\mu=-1/2-\varepsilon$ for $0<\varepsilon \ll 1$. 
\end{proof}
		
		\begin{proof}[Proof of Corollary \ref{cor14}]
		The claim follows as the isomorphism \eqref{K} from the proof of Corollary \ref{thm12}.
		\end{proof}
		
In the context of Corollary \ref{cor14}, we do not know whether the following is an isomorphism
\[K\left(\mathcal{DT}(d-d')\otimes \mathcal{PT}_{\mathbb{C}^3}^{\rm{red}}(m, d')\right)\cong 
		K\left(\mathcal{DT}(d-d')\right)\otimes K\left(\mathcal{PT}_{\mathbb{C}^3}^{\rm{red}}(m, d')\right).\]
To obtain a result analogous to the one of Corollary \ref{thm12}, we use equivariant K-theory, see \cite[Section 4]{PT0} for similar constructions and computations.
Let $T\cong (\mathbb{C}^*)^2$ be the natural two dimensional torus 
acting on $\mathbb{C}^2$.
We let $T$ act on $\mathbb{C}^3$ as follows: \[(t_1, t_2)\cdot(x, y, z)=(t_1x, t_2y, t_1^{-1}t_2^{-1}z).\]
It induces the action on 
$V^{\oplus 2} \oplus V^{\vee} \oplus \mathfrak{g}^{\oplus 3} \otimes \mathbb{C}^{I_m}$
by 
\begin{align*}
    (t_1, t_2) \cdot (u_1, u_2, v, A, B, C, 
    \alpha_{ij})=(u_1, u_2, v, t_1 A, t_2 B, 
    t_1^{-1}t_2^{-1}C, t_1^{-i}t_2^{-j}\alpha_{ij}). 
\end{align*}
The function (\ref{W:formula}) is invariant 
under the above $T$-action. 
We consider the equivariant categories $\mathcal{DT}^{\rm{red}}_{\mathbb{C}^3, T}(m, d)$ and $\mathcal{PT}^{\rm{red}}_{\mathbb{C}^3, T}(m, d)$ and we denote their Grothendieck groups by 
\[K_T(\mathcal{DT}_{\mathbb{C}^3}^{\rm{red}}(m, d))\text{ and }K_T(\mathcal{PT}_{\mathbb{C}^3}^{\rm{red}}(m, d)).\]
Let $\mathbb{K}:=K_T(\text{pt})$ and let $\mathbb{F}:=\text{Frac}\,\mathbb{K}$. For $V$ a $\mathbb{K}$-module, let $V_\mathbb{F}:=V\otimes_{\mathbb{K}}\mathbb{F}$. 

\begin{cor}\label{cor15}
Let $d\in \mathbb{N}$. Then there is an isomorphism of $\mathbb{F}$-vector spaces: \begin{align*}
    K_T\left(\mathcal{DT}_{\mathbb{C}^3}^{\rm{red}}(m, d)\right)_\mathbb{F}\cong
    \bigoplus_{d'=0}^d K_T\left(\mathcal{DT}(d-d')\right)_{\mathbb{F}}\otimes_{\mathbb{F}} K_T\left(\mathcal{PT}_{\mathbb{C}^3}^{\rm{red}}(m, d')\right)_\mathbb{F}.
\end{align*}
\end{cor}

\begin{proof}
As in the proof of Corollary \ref{cor14}, there is an isomorphism of $\mathbb{F}$-vector spaces:
\[K_T\left(\mathcal{DT}_{\mathbb{C}^3}^{\rm{red}}(m, d)\right)_\mathbb{F}\cong
    \bigoplus_{d'=0}^d K_T\left(\mathcal{DT}(d-d')\otimes \mathcal{PT}_{\mathbb{C}^3}^{\rm{red}}(m, d')\right)_\mathbb{F}.\]
It suffices to show that there is a Künneth isomorphism:
\begin{align}\label{kunnethpot}
    &K_T\left(\mathcal{DT}(d-d')\otimes \mathcal{PT}_{\mathbb{C}^3}^{\rm{red}}(m, d')\right)_\mathbb{F}\cong\\ &K_T\left(\mathcal{DT}(d-d')\right)_{\mathbb{F}}\otimes_{\mathbb{F}} K_T\left(\mathcal{PT}_{\mathbb{C}^3}^{\rm{red}}(m, d')\right)_\mathbb{F}.\notag
\end{align}
The category $\mathcal{DT}_T(d)$ has a semiorthogonal decomposition with summands Hall products of  $\mathbb{S}_T(e)_v$ for various $(e, v)\in \mathbb{N}\times\mathbb{Z}$ by \cite[Theorem 1.1]{PT0}. 
Let $\mathbb{S}^{\text{gr}}_T(d)_w$ be the graded category of matrix factorizations, defined similarly to $\mathbb{S}_T(d)_w$, where the linear map corresponding to $X$ is scaled by $2$.
Then $\mathbb{S}^{\text{gr}}_T(d)_w$ and $\mathbb{S}_T(d)_w$ have the same Grothendieck group \cite[Corollary~3.13]{T4}. 
Further, $\mathbb{S}^{\text{gr}}_T(d)_w$
appears as summands in $D^b_T(\mathscr{C}(d))_w$ \cite[Corollary~3.3]{P2}, where $\mathscr{C}(d)$ is the (derived) stack of pairs of commuting matrices of size $d$. Consider the analogous graded category 
$\mathcal{PT}^{\rm{red}, \rm{gr}}_{\mathbb{C}^3, T}(m, d)$ with the same Grothendieck group as $\mathcal{PT}^{\rm{red}}_{\mathbb{C}^3, T}(m, d)$.
By \cite[Theorem 2.10]{halp} and the Koszul equivalence, see for example \cite[Theorem 2.3.3]{T},
the category $\mathcal{PT}^{\rm{red}, \rm{gr}}_{\mathbb{C}^3, T}(m, d)$ is admissible in $D^b_T(	\mathfrak{T}_{\mathbb{C}^2}^{\rm{red}}(m, d))$.
By the argument in the proof of Corollary \ref{thm12}, it suffices to check the Künneth isomorphism 
\begin{equation}\label{kunnethstacks}
G_T\left(\mathscr{D}\times\mathfrak{T}_{\mathbb{C}^2}^{\rm{red}}(m, d')\right)_\mathbb{F}\cong
    G_T\left(\mathscr{D}\right)_\mathbb{F}\otimes_\mathbb{F} 
    G_T\left(\mathfrak{T}_{\mathbb{C}^2}^{\rm{red}}(m, d')\right)_\mathbb{F},
\end{equation} where $\mathscr{D}$ is the product of finitely many commuting stacks $\mathscr{C}(e)$ for various $e$.
The statement follows from the localization theorem in K-theory (see \cite{MR1297442}) with respect to $T$ because $\mathscr{D}^T=\text{pt}/L$, where $L$ is a product of finitely many groups $GL(e)$. 
\end{proof}

	\bibliographystyle{amsalpha}
\bibliography{math}

\textsc{\small Tudor P\u adurariu: Columbia University, 
Mathematics Hall, 2990 Broadway, New York, NY 10027, USA.}\\
\textit{\small E-mail address:} \texttt{\small tgp2109@columbia.edu}\\

\textsc{\small Yukinobu Toda: Kavli Institute for the Physics and Mathematics of the Universe (WPI), University of Tokyo, 5-1-5 Kashiwanoha, Kashiwa, 277-8583, Japan.}\\
\textit{\small E-mail address:} \texttt{\small yukinobu.toda@ipmu.jp}\\

\end{document}